\documentclass[12pt, a4paper]{amsart}
\usepackage{amscd,amsmath,amssymb,amsthm,amsfonts}
\usepackage[alphabetic]{amsrefs}
\usepackage[shortlabels]{enumitem}

\usepackage{xcolor}
\usepackage[hypertexnames=true, citecolor = green, colorlinks = false, linkcolor = red, linktoc = none, pdffitwindow = false, urlbordercolor = white]{hyperref}%

\usepackage{amsthm}
\def\Aut{\operatorname{Aut}}
\def\End{\operatorname{End}}
\def\newspan{\operatorname{span}}

\def\id{\operatorname{id}}

\def\id{\operatorname{id}}

\def\op{\operatorname{op}}

\def\C{\mathbb{C}}

\def\N{\mathbb{N}}
\def\Z{\mathbb{Z}}

\def\TT{\mathcal{T}}
\def\LL{\mathcal{L}}
\def\OO{\mathcal{O}}
\def\KK{\mathcal{K}}

\def\AA{\mathcal{A}}

\def\DD{\mathcal{D}}
\def\PP{\mathcal{P}}

\def\FF{\mathcal{F}}

\def\EE{\mathcal{E}}
\def\JJ{\mathcal{J}}
\def\QQ{\mathcal{Q}}

\def\NT{\mathcal{N}\mathcal{T}}

\def\gxp{G\rtimes_\theta P}

\newtheorem{thm}{Theorem}[section]
\newtheorem{cor}[thm]{Corollary}

\newtheorem{lemma}[thm]{Lemma}

\newtheorem{prop}[thm]{Proposition}

\theoremstyle{definition}
\newtheorem{definition}[thm]{Definition}
\newtheorem{notation}[thm]{Notation}

\theoremstyle{remark}
\newtheorem{remark}[thm]{Remark}

\newtheorem{example}[thm]{Example}

\numberwithin{equation}{section}

\begin{document}

\title{\texorpdfstring{$C^*$-algebras of algebraic dynamical systems and right LCM semigroups}{C*-algebras of algebraic dynamical systems and right LCM semigroups}}

\author{Nathan Brownlowe}
\address{School of Mathematics and Statistics \\ The University of Sydney \\ NSW 2006\\ Australia}
\email{nathan.brownlowe@sydney.edu.au}

\author{Nadia S.~Larsen}
\address{Department of Mathematics \\ University of Oslo \\ PO BOX 1053 Blindern\\ 0316 Oslo\\ Norway}
\email{nadiasl@math.uio.no}

\author{Nicolai Stammeier}
\address{Department of Mathematics \\  University of Oslo \\ PO BOX 1053 Blindern\\ 0316 Oslo\\ Norway}
\email{nicolsta@math.uio.no}

\begin{abstract}
We introduce algebraic dynamical systems, which consist of an action of a right LCM semigroup by injective endomorphisms of a group. To each algebraic dynamical system we associate a $C^*$-algebra and describe it as a semigroup $C^*$-algebra. As part of our analysis of these $C^*$-algebras we prove results for right LCM semigroups. More precisely we discuss functoriality of the full semigroup $C^*$-algebra and compute its $K$-theory for a large class of semigroups. We introduce the notion of a Nica-Toeplitz algebra of a product system over a right LCM semigroup, and show that it provides a useful alternative to study algebraic dynamical systems.
\end{abstract}

\date{16 December, 2016. Revised: 25 April 2017}
\maketitle

\section{Introduction}\label{sec: intro}
\noindent Dynamical systems arising from injective endomorphisms of countable, discrete groups have
provided an important source of algebraic data from which to build
$C^*$-algebras. The case of group automorphisms can by now be considered classical, and is modelled
$C^*$-algebraically by the much-studied crossed product. More recent work in
the area has focussed on injective endomorphisms which are not surjective. The
study of $C^*$-algebras associated to a single injective endomorphism of a group with finite
cokernel includes Hirshberg's work \cite{Hir} on endomorphisms of amenable groups,
and Cuntz and Vershik's work \cite{CV} on endomorphisms of abelian groups, while the situation with an injective endomorphism with infinite cokernel was considered by Vieira \cite{Vie}.

Through recent work in \cite{BLS1,Sta1} the authors have sought to build a
unifying framework for the above constructions, while at the same time
significantly broadening the scope of the theory. In \cite{Sta1} the
third-named author generalised the single-endomorphism setting to actions
of semigroups of injective endomorphisms via the notion of an
{\em irreversible algebraic dynamical system}, which consists of an action
of a countably generated, free abelian monoid $P$ (with identity) by injective group
endomorphisms subject to an independence condition which originated from \cite{CV}. A construction of very similar flavour was proposed by the first two authors, see \cite[Example 5.10]{HLS} for details. The $C^*$-algebra from \cite{Sta1} associated to an irreversible algebraic dynamical system has a presentation by generators and relations that encodes the action, and can be viewed alternatively  as a Cuntz-Nica-Pimsner algebra.
In \cite{BLS1} the authors looked at similar actions $\theta:P\curvearrowright G$, and
studied them via the semidirect product semigroup $\gxp$. The focus of \cite{BLS1} is on the semigroup
$C^*$-algebra, in the sense of Li \cite{Li1}, associated to $\gxp$.

In many cases, the semidirect products $\gxp$ from \cite{BLS1} are examples of {\em right LCM
semigroups}. These are left cancellative semigroups with the property that principal right ideals
intersect to be either empty or another principal right ideal. Their $C^*$-algebras, both the semigroup $C^*$-algebra and a boundary-quotient type $C^*$-algebra, have  garnered recent attention \cite{Nor,BRRW,Star}. Right LCM semigroups lead to a
tractable and yet rich class of semigroup $C^*$-algebras that includes the $C^*$-algebras associated
to quasi-lattice ordered pairs as introduced by Nica \cite{Nic}, and the Toeplitz type
$C^*$-algebras associated to a self-similar actions in \cite{LRRW} as a counterpart to the Cuntz-Pimsner algebras from \cite{Nek}. In \cite{BLS1}, the authors pursued the study of the semigroup $C^*$-algebras of right LCM semigroups by establishing uniqueness results.

This paper serves two major purposes. The first, and our original
motivation for this work, is to significantly extend the work of \cite{Sta1}
by introducing the notion of an algebraic dynamical system and studying $C^*$-algebras that encode the action. An \emph{algebraic dynamical system} consists of an action
$\theta:P\curvearrowright G$ of a right LCM semigroup (with identity) $P$ by injective
endomorphisms of a group $G$ which harmonise with the structure of the principal right ideals in $P$. This setting is more general than the one considered in \cite{Sta1} and gives a large class of new examples.
To every algebraic dynamical system $(G,P,\theta)$ we associate a
$C^*$-algebra $\AA[G,P,\theta]$ that is universal for a unitary representation of $G$ and
a representation of $P$ by isometries satisfying relations which encode the dynamics. These relations lead to a Wick ordering on $\AA[G,P,\theta]$. Our first main result says that
$\AA[G,P,\theta]$ is canonically isomorphic to the full semigroup $C^*$-algebra of $\gxp$.
This explains and motivates our renewed interest in the study of the full semigroup $C^*$-algebra
of right LCM semigroups. Therefore, a continuation of the work in \cite{BLS1} in order to build a general theory of semigroup $C^*$-algebras associated to right LCM semigroups constitutes the second major purpose of this paper.

Conceptionally, $\AA[G,P,\theta]$ resembles a Toeplitz type crossed product for a generalised
Exel-Larsen system naturally arising from $(G,P,\theta)$. To make this idea precise we appeal to Fowler's product systems of right-Hilbert bimodules over semigroups \cite{Fow2}. Explicitly, we show that his theory, that was originally developed for quasi-lattice ordered pairs, can be taken a considerable step further to the setting of right LCM semigroups.

Motivated by our analysis of the $C^*$-algebra associated to algebraic
dynamical systems, we prove a number of results for right LCM semigroups:
\begin{itemize}[leftmargin=1cm]
\item[(1)] We investigate the functoriality of $S\mapsto C^*(S)$ and determine the right notion of a morphism for the category with objects right LCM semigroups so that the construction of the full semigroup $C^*$-algebra becomes functorial. We also characterise surjectivity of the induced morphism, and injectivity in the case that both left regular representations are faithful.

\item[(2)] We use the machinery of Cuntz, Echterhoff and Li \cite{CEL1} to show that, for a large
class of right LCM semigroups $S$, the $K$-theory of $C^*(S)$ is given by the $K$-theory of
$C^*_r(S^*)$, where $S^*$ denotes the group of units in $S$.

\item[(3)] We extend the notion of a compactly aligned product system of right-Hilbert bimodules, and
of a Nica covariant representation of such a product system, from the quasi-lattice ordered setting first
introduced by Fowler \cite{Fow2} to the setting of right LCM semigroups. This leads to the definition of a
Nica-Toeplitz algebra for product systems over right LCM semigroups.
\end{itemize}

Our description of $\AA[G,P,\theta]$ as the semigroup $C^*$-algebra
$C^*(\gxp)$ of the right LCM semigroup $\gxp$ allows us to apply the general theory from (1) and (2). We thus characterise
the morphisms between algebraic dynamical systems which induce morphisms of their
$C^*$-algebras, and we compute the $K$-theory of $\AA[G,P,\theta]$ for a large
class of algebraic dynamical systems $(G,P,\theta)$. We also apply the general theory built
in (3) to the setting of algebraic dynamical systems, where we construct a
compactly aligned product system $M$ from an algebraic dynamical system $(G,P,\theta)$ whose
Nica-Toeplitz algebra $\NT(M)$ is canonically isomorphic to $\AA[G,P,\theta]$. This allows us to conclude that, for right LCM semigroups $S$, $C^*(S)$ can be modelled as the Nica-Toeplitz algebra of the product system over $S$ with fibres $\C$. We would like to stress the point that there is an abundance of compactly aligned product systems $M$, built from algebraic dynamical systems, for which the left action on some (or all) of the fibres of $M$ is not by generalised compacts, as has frequently been the case in the context of quasi-lattice ordered pairs.

We finish with a section on examples, the most basic of which is the
system $(\Z,P,\cdot)$ in which $P\subset \Z^\times$ is generated by a family of relatively
prime nonzero integers acting on $\Z$ by multiplication. Despite the
straightforward nature of these systems, we show that they provide an
interesting class of $C^*$-algebras that warrant further study. Other
examples of algebraic dynamical systems include the more general $(R,P,\cdot)$
in which $R$ is a ring of integers in a number field and $P$ a right LCM
subsemigroup of $R^\times$ acting by multiplication, and shift systems
$(\bigoplus_P G_0,P,\theta)$ in which $G_0$ is countable, $P$ right LCM, and
$\theta$ the natural shift action on $\bigoplus_P G_0$.

The structure of the paper is as follows: in Section~\ref{sec: ADS and their C*-algebras} we introduce algebraic dynamical systems $(G,P,\theta)$ and their $C^*$-algebras $\AA[G,P,\theta]$. In
Section~\ref{sec: semigroup description} we address the question of functoriality
of $S\mapsto C^*(S)$, and in Section~\ref{sec: semigroup description for ADS} we prove that
$\AA[G,P,\theta]\cong C^*(\gxp)$, and apply the results of Section~\ref{sec: semigroup description} to
the context of $\AA[G,P,\theta]$. In Section~\ref{sec: K-theory} we prove that for a large
class of right LCM $S$ we have $K_*(C^*(S))\cong K_*(C_r^*(S^*))$, where $C_r^*(S^*)$ is the reduced
$C^*$-algebra of the group of units $S^*$ in $S$. In Section~\ref{sec: right LCM product systems}
we introduce the Nica-Toeplitz algebra for a product system over a right LCM semigroup, and
then in Section~\ref{sec: product system for ADS} we describe $\AA[G,P,\theta]$ as such a Nica-Toeplitz
algebra. We finish in Section~\ref{sec: examples}
with a discussion of examples of algebraic dynamical systems and their $C^*$-algebras.
The flow of the paper is a steady back and forth between algebraic dynamical systems
and general right LCM semigroups.
We feel that this arrangement reflects in a proper way the fruitful interplay between the two topics in the context of $C^*$-algebraic constructions.

We would like to point out that algebraic dynamical systems are not the same as \emph{algebraic dynamics}, a terminology going back to \cite{Sch1}. The objects of interest in this area are actions of countable, discrete groups on compact (abelian) groups, see \cite{LS} for a recent survey of highly interesting connections to other fields. We remark that algebraic dynamical systems $(G,P,\theta)$ can be thought of as algebraic dynamics provided that $G$ is abelian and $P$ is a group. We hope that our choice does not cause confusion, but rather stimulates an examination of the connections between these two concepts.

While this manuscript was under review, its contents served as fruitful grounds for \cite{BS1}, where many of the results from \cite{Sta1} are extended to algebraic dynamical systems, building heavily on \cite{Star}. This lead to a generalisation of the \emph{boundary quotient diagram} of \cite{BaHLR} from $\N\rtimes\N^\times$ to a broad class of right LCM semigroups, see \cite{Sta3}. The insights into the internal structure of right LCM semigroups gained from this construction then allowed the authors (in collaboration with Zahra Afsar) to determine the structure of equilibrium states on $C^*(S)$ for a natural dynamics for a surprisingly large class of right LCM semigroups $S$, encompassing $\N\rtimes \N^\times$, dilation matrices, self-similar group actions, Baumslag-Solitar monoids, and many arising as semidirect products from algebraic dynamical systems, see \cite{ABLS}.\vspace*{0mm}\\

\noindent\emph{Acknowledgements:} The present work has advanced substantially during multiple visits of the first and the third author to Oslo in 2013 and 2014, and we would like to thank the group in Oslo for their hospitality. This work has also benefited from a joint week spent at the BIRS mini-workshop 13w5152 in November 2013 and we are grateful to the organisers Alan Carey and Marcelo Laca for inviting the three of us. The third author acknowledges support provided by DFG through SFB $878$ and by ERC through AdG $267079$. We thank S. Echterhoff for directing our attention to \cite{CCH}.

\section{\texorpdfstring{Algebraic dynamical systems and their $C^*$-algebras}{Algebraic dynamical systems and their C*-algebras}}\label{sec: ADS and their C*-algebras}

\noindent In \cite[Definition 1.5]{Sta1}, the third author introduced the
notion of \emph{irreversible algebraic dynamical systems} and studied
associated $C^*$-algebras, which under mild assumptions turn out to be unital UCT Kirchberg algebras. Recall that an irreversible algebraic dynamical system
$(G,P,\theta)$ is given by a countable, discrete group $G$, a countably generated, free abelian semigroup $P$ with identity, and a $P$-action $\theta$ on $G$ by injective group endomorphisms such that
\[
\theta_p(G) \cap \theta_q(G) = \theta_{pq}(G) \text{ if and only if } pP
\cap qP = pqP.
\]

Note that due to the above displayed condition and injectivity of $\theta_p$, the group
$G$ needs to be infinite whenever $P$ is nontrivial. Moreover, the only group
automorphism of $G$ present in such a dynamical system is $\id_G$.\vspace*{0.5mm}\\

\noindent We are now
going to provide a vast generalisation of the concept of an irreversible
algebraic dynamical system $(G,P,\theta)$, where $P$ is allowed to be any
countable right LCM semigroup with identity, for example a group, and relations of
the images of $\theta_p$ and $\theta_q$ no longer imply relations for $pP$ and
$qP$. Recall that a semigroup $S$ is right LCM if it is left cancellative and
for any $p, q$ in $S$, the intersection
of principal right ideals  $pS\cap qS$ is either empty or of the form $rS$ for
some $r\in S$, see \cite{BRRW}.

All identities for semigroups will be denoted by $1$. If $S$ is a semigroup with identity, we let $S^*$ denote the group of invertible elements in $S$.

\begin{definition}\label{def:ADS}
An {\em algebraic dynamical system} is a triple $(G,P,\theta)$ consisting of a
countable, discrete group $G$, a countable, right LCM semigroup $P$ with
identity, and an action $\theta$ of $P$ by injective group endomorphisms of
$G$ such that 
for all $p,q \in P$:
\begin{equation}\label{eq:preserving the order}
\theta_p(G) \cap \theta_q(G) = \theta_r(G) \text{ whenever } r\in P \text{
satisfies } pP \cap qP = rP.
\end{equation}
\end{definition}

\noindent Following \cite{BLS1} we say that $\theta$
\emph{respects the order on} $P$ whenever \eqref{eq:preserving the order}  is
satisfied. Note that if $r' \in P$ satisfies $rP=r'P$,
then there is $x \in P^*$ such that $r' = rx$. Hence one has $\theta_{r'}(G) =
\theta_r(\theta_x(G)) = \theta_r(G)$ as $\theta_x$ is a group automorphism of
$G$.

The class of algebraic dynamical systems contains various types of examples, see Section~\ref{sec: examples}.
In particular, it includes the irreversible algebraic dynamical systems from  \cite{Sta1}. We showed in
\cite[Proposition 8.2]{BLS1} that $\gxp$ is a right LCM semigroup whenever $(G, P, \theta)$ is a triple as in
Definition~\ref{def:ADS}. In later sections we shall exploit this connection, but for the moment we focus our attention on the triple
$(G, P,\theta)$.

\begin{definition}\label{def: NT}
For an algebraic dynamical system $(G,P,\theta)$ let $\AA[G,P,\theta]$ be the universal
$C^*$-algebra generated by a unitary representation $u$ of $G$ and an
isometric representation $s$ of $P$ satisfying
\[\begin{array}{lrcl}
\textup{(A1)} & s_p u_g &=& u_{\theta_p(g)}s_p \text{ for all }p\in P,
g\in G,
\text{ and }\vspace*{2mm}\\
\textup{(A2)} & s_p^*u_g^{\phantom{*}}s_q^{\phantom{*}} &=&
\begin{cases}
u_{k}^{\phantom{*}}s_{p'}^{\phantom{*}}s_{q'}^*u_{\ell}^* &\text{if $pP \cap qP = rP,\, pp'=qq'=r$}\\
 & \text{and $g=\theta_p(k)\theta_q(\ell^{-1})$ for some $k,\ell\in G$,}\\
0& \text{otherwise.}
\end{cases}
\end{array}\]
\end{definition}

\noindent We need to check that (A2) does not depend on the choice of
$r$ or $(k,\ell)$. If $r' \in P$ satisfies $r'P = rP$, then there exists $x \in P^*$
such that $r' = rx$. If $p'',q''\in P$ satisfy $pp''=qq''=r'$, then $p''=p'x, q''=q'x$,
and hence $s_{p''}s_{q''}^* = s_{p'}s_xs_x^*s_{q'}^* = s_{p'}s_{q'}^*$.

Now suppose $k_1,\ell_1 \in G$ satisfy
$\theta_p(k_1)\theta_q(\ell_1^{-1}) = g = \theta_p(k)\theta_q(\ell^{-1})$ and
$pP
\cap qP
= rP$ for some $r \in P$. But $G$ is a group, so we have
$\theta_p(k^{-1}k_1) = \theta_q(\ell^{-1}\ell_1)$. As $\theta$ respects the
order on $P$, this element belongs to $\theta_r(G)$. Since $\theta_p$
and
$\theta_q$ are injective, we get $k^{-1}k_1 =
\theta_{p'}(k_2)$ and $\ell^{-1}\ell_1 = \theta_{q'}(\ell_2)$ for some
$k_2,\ell_2 \in G$. Note that
$\theta_p(k^{-1}k_1) = \theta_q(\ell^{-1}\ell_1)$ is the same as
$\theta_r(k_2)
= \theta_r(\ell_2)$, and injectivity of $\theta_r$ then implies that
$k_2=\ell_2$. Using (A1) we conclude
\[
u_{k_1}s_{p'}s_{q'}^*u_{\ell_1}^* =
u_{k}u_{k^{-1}k_1}s_{p'}s_{q'}^*u_{\ell^{-1}\ell_1}^*u_{\ell}^*
=u_{k}s_{p'}u_{k_2}u_{\ell_2}^*s_{q'}^*u_{\ell}^*
= u_{k}s_{p'}s_{q'}^*u_{\ell}^*.
\]

\begin{remark}\label{rem: justification for (A2)}
We chose (A2) as one of our defining relations of $\AA[G,P,\theta]$ because the expressions of the form $s_p^*u_gs_q$ are precisely the ones which need to be expressible in terms of $u_{g_1}s_{p_1}s_{p_2}^*u_{g_2}^*, g_i \in G, p_i \in P$ in order to get a \emph{Wick ordering} on $\AA[G,P,\theta]$, i.e. monomials in the generators where all starred generators appear on the right of all non-starred generators, see \cite{JSW}.
\end{remark}

\noindent We like to think of $\AA[G,P,\theta]$ as a Toeplitz type crossed product of $C^*(G)$ by the $P$-action $\theta$. If one looks into the literature for semigroup crossed products, say by abelian semigroups, one encounters a covariance condition that deals with the analogue of $s_p^*u_gs_p$ by means of a transfer operator. This was introduced by Ruy Exel in \cite{Exe1} for $\N$ and was extended to abelian semigroups by the second author in \cite{Lar}. This development gave rise to the notion of Exel-Larsen systems. In fact, for each algebraic dynamical system, one can associate a transfer operator $L$ to the action $\alpha:C^*(G) \to C^*(G)$ induced by $\theta$. Explicitly,  $L$ is the semigroup homomorphism from the opposite semigroup $P^{\text{op}}$ of $P$ to the semigroup of unital, positive, linear maps $C^*(G) \to C^*(G)$ given by $L_p(\delta_g) = \chi_{\theta_p(G)}(g) \delta_{\theta_p^{-1}(g)}$, where $\delta_g$ denote the generating unitaries in $C^*(G)$. In this way, we can regard $\AA[G,P,\theta]$ as a Toeplitz type crossed product for the generalised Exel-Larsen system $(G,P,\alpha,L)$, see Section~\ref{sec: product system for ADS} for details.

The step of moving from abelian to more general semigroups $P$ requires taking care of expressions of the form $s_p^*s_q$ as well. If $P$ happens to be right LCM, Nica covariance as it appeared first in \cite{Nic} provides a natural means to replace $s_p^*s_q$ by $s_{q'}s_{p'}^*$ if $pP \cap qP = rP$ with $r=pp'=qq'$ and declare it to vanish otherwise.

Altogether we see that (A2) combines the idea of a covariance condition coming from a transfer operator for $\theta$ and Nica covariance for the semigroup $P$. Note, however, that (A2) cannot be replaced by these two conditions in general, e.g. they say nothing about terms of the form $s_p^*u_gs_q$ with $g \neq 1$ and $p,q \in P\setminus P^*$ such that $rP \supset pP \cup qP$ implies $r \in P^*$ for all $r \in P$.

For each $g\in G$, $p\in P$ we denote the projection $u_gs_ps_p^*u_g^*$ by $e_{(g,p)}$. In the following lemma we show that relation (A2) is equivalent to a more familiar looking relation involving the product of these
projections $e_{(g,p)}$.

\begin{lemma}\label{lem: A2 and projs}
Let $(G,P,\theta)$ be an algebraic dynamical system, $u$ a unitary
representation of $G$ and $s$ an isometric representation of $P$
satisfying
(A1). Then $u$ and $s$ satisfy (A2) if and only if
\begin{equation}\label{eq: prod of projs}
e_{(g,p)}e_{(h,q)}=
\begin{cases}
e_{(g\theta_p(k),r)} & \text{if $pP\cap qP=rP$ for some $r\in P$
and
}\\
 & \text{$g\theta_p(k)\in h\theta_q(G)$ for
some $k\in G$,}\\
0 & \text{otherwise.}
\end{cases}
\end{equation}
\end{lemma}

\begin{proof}
Suppose (A2) holds. Since $e_{(g,p)}e_{(h,q)} =
u_gs_p(s_p^*u_{g^{-1}h}s_q)s_q^*u_h^*$, we know from (A2) that
$e_{(g,p)}e_{(h,q)}$ is
zero unless $pP \cap qP = rP$ and $g^{-1}h=\theta_p(k)\theta_q(\ell^{-1})$
for
some $k,\ell\in G$. If these conditions hold, and $p',q'\in P$ with
$pp'=qq'=r$,
then (A1) and (A2) give
\[
e_{(g,p)}e_{(h,q)} = u_gs_p(s_p^*u_{g^{-1}h}s_q)s_q^*u_h^* =
u_gs_pu_ks_{p'}s_{q'}^*u_{\ell}^*s_q^*u_h^* =
u_{g\theta_p(k)}s_rs_r^*u_{h\theta_q(\ell)}^* = e_{(g\theta_p(k),r)}.
\]
Now \eqref{eq: prod of projs} holds because $g^{-1}h=\theta_p(k)\theta_q(\ell^{-1})$ for
some $k,\ell\in G$ if and only if $g\theta_p(k)\in h\theta_q(G)$ for some $k\in G$.

Suppose \eqref{eq: prod of projs} holds. Let $g \in G$ and $p,q \in P$. Then $s_p^*u_gs_q = s_p^*e_{(1,p)}e_{(g,q)}u_gs_q$ vanishes according to \eqref{eq: prod of projs} unless $pP \cap qP = rP$ for some $r \in P$ and $1\theta_p(k) = g\theta_q(\ell)$ for some $k,\ell \in G$. Note that this is the same condition as in (A2). If we let $p',q' \in P$ satisfy $pp'=qq'=r$, we arrive at
\[s_p^*u_gs_q = s_p^*e_{(1,p)}e_{(g,q)}u_gs_q = s_p^*u_{\theta_p(k)}s_{pp'}s_{qq'}^*u_{g\theta_q(\ell)}^*u_gs_q
= u_ks_{p'}s_{q'}^*u_{\ell}^*,\]
which establishes (A2).
\end{proof}

\begin{notation}\label{not: transversal}
Given $p \in P$, we will refer to a complete set of representatives for $G/\theta_p(G)$
as a \emph{transversal} of $G/\theta_p(G)$ and denote it by $T_p$.
\end{notation}

\begin{remark}\label{rem: mut orthog e's}
Note that \eqref{eq: prod of projs} implies that for each $p\in P$  and each transversal $T_p$ of  $G/\theta_p(G)$, the
projections $\{e_{(g,p)}\mid g\in T_p\}$ form
a collection of mutually orthogonal projections.
\end{remark}

\noindent  We would like to discuss the question of functoriality for the mapping
$(G,P,\theta) \mapsto \AA[G,P,\theta]$. To this end, we introduce a natural notion of a
morphism between algebraic dynamical systems.

\begin{definition}\label{def:morphism of ADS}
Suppose $(G_1,P_1,\theta_1)$ and $(G_2,P_2,\theta_2)$ are algebraic dynamical
systems. A morphism from $(G_1,P_1,\theta_1)$ to $(G_2,P_2,\theta_2)$ is a
pair $(\phi_G,\phi_P)$, where
\begin{enumerate}[(i)]
\item $\phi_G: G_1 \to G_2$ is a group homomorphism,
\item $\phi_P: P_1 \to P_2$ is an identity preserving semigroup homomorphism, and
\item $\phi_G \circ \theta_{1,p} = \theta_{2,\phi_P(p)} \circ \phi_G$
holds for all $p \in P_1$.
\end{enumerate}
\end{definition}

\noindent The proof of the following lemma is straightforward and therefore omitted.

\begin{lemma}\label{lem:morphism of ADS induces sgp hom of GxP}
Suppose $(G_1,P_1,\theta_1)$ and $(G_2,P_2,\theta_2)$ are algebraic dynamical
systems. If $(\phi_G,\phi_P)$ is a morphism from $(G_1,P_1,\theta_1)$ to
$(G_2,P_2,\theta_2)$, then $(g,p) \mapsto (\phi_G(g),\phi_P(p))$ defines
an identity preserving semigroup homomorphism $\phi_G \rtimes \phi_P: G_1 \rtimes_{\theta_1}
P_1 \to G_2 \rtimes_{\theta_2} P_2$.
\end{lemma}

\begin{remark}\label{rem:morphisms of ADS vs. sgp hom of GxP}
By Lemma~\ref{lem:morphism of ADS induces sgp hom of GxP}, every morphism
between algebraic dynamical systems gives rise to a homomorphism between the
corresponding semigroups. The converse is false in general. More precisely,
for given algebraic dynamical systems $(G_1,P_1,\theta_1)$ and
$(G_2,P_2,\theta_2)$, a homomorphism $\varphi: G_1 \rtimes_{\theta_1} P_1 \to
G_2 \rtimes_{\theta_2} P_2$ is induced by a morphism of the algebraic
dynamical systems if and only if $\varphi(G_1 \times \{1\}) \subset G_2 \times
\{1\}$ and $\varphi(\{1\}\times P_1) \subset \{1\}\times P_2$.
\end{remark}

\noindent It is a natural question whether  every morphism $(\phi_G,\phi_P):
(G_1,P_1,\theta_1) \to (G_2,P_2,\theta_2)$ induces a $*$-homomorphism $\varphi_G
\rtimes \varphi_P:\AA[G_1,P_1,\theta_1] \to \AA[G_2,P_2,\theta_2]$. We shall postpone an answer to this question until the end of Section~4.

\section{\texorpdfstring{$C^*$-algebras of right LCM semigroups and functoriality}{C*-algebras of right LCM semigroups and functoriality}}\label{sec: semigroup description}

\noindent In \cite{Li1}, Li investigated the functoriality of the assignment $S\mapsto C^*(S)$ in the context of $ax+b$-semigroups over integral domains, and applied his findings to the Toeplitz type $C^*$-algebras associated to number fields in \cite{CDL}. He remarked that functoriality is not likely to be easily described for arbitrary $S$. Here we show that it can be successfully approached in case that $S$ is a right LCM semigroup with identity.

\medskip
\noindent We first recall Li's construction of the full and the reduced $C^*$-algebra associated to a
discrete left cancellative semigroup $S$. A set $X\subseteq S$ is a right ideal if it is closed under right
multiplication with any element of $S$. For each right ideal $X$, the sets
\[
 pX = \{px \mid x \in X\}\quad \text{and}\quad p^{-1} X = \{y\in S \mid py \in X\}.
\]
are also right ideals. Li \cite[p.4]{Li1} defines $\JJ(S)$ to be the
smallest family of right ideals of $S$ satisfying
\begin{enumerate}[(a)]
 \item $S, \emptyset \in \JJ(S)$; and
 \item $X \in \JJ(S)$ and $p \in S$ implies $pX$ and $p^{-1} X \in \JJ(S)$.
\end{enumerate}
The elements of $\JJ(S)$ are called {\em constructible right ideals}. The
general form of a constructible right ideal is given in \cite[Equation~(5)]{Li1}.
We note that $\JJ(S)$ is also closed under finite intersections, a fact that can be derived from (a) and (b) using $pS \cap qS = p(p^{-1}(qS))$.

\begin{definition}\label{def: independence}
The set of constructible right ideals $\JJ(S)$ is called {\em independent} if
for every $X,X_1,\dots,X_n\in\JJ(S)$ we have
\[
X_j\varsubsetneq X\text{ for all $1\le j\le n$}\Longrightarrow \bigcup_{j=1}^n
X_j\varsubsetneq X.
\]
Alternatively, $\JJ(S)$ is independent if $\bigcup_{j=1}^n X_j=
X\Longrightarrow X_j=X$ for some $j$.
\end{definition}

\noindent Let us recall Li's definition of the full semigroup $C^*$-algebra for $S$.

\begin{definition}\label{def: Li's full algebra}
Let $S$ be a discrete left cancellative semigroup. The {\em full semigroup
$C^*$-algebra} $C^*(S)$ is the universal $C^*$-algebra generated by isometries
$(v_p)_{p\in S}$ and projections $(e_X)_{X\in\JJ(S)}$ satisfying
\begin{enumerate}
\item[(L1)] $v_pv_q=v_{pq}$;
\item[(L2)] $v_pe_Xv_p^*=e_{pX}$;
\item[(L3)] $e_\emptyset=0$ and $e_S=1$; and
\item[(L4)] $e_Xe_Y=e_{X\cap Y}$,
\end{enumerate}
for all $p,q\in S$, $X,Y\in\JJ(S)$. Note that (L2) and (L3) give $v_pv_p^*=e_{pS}$ for all $p\in S$.
\end{definition}

\noindent The reduced semigroup $C^*$-algebra $C_r^*(S)$ is defined by the left regular representation of $S$ on $l^2(S)$: if $\{\varepsilon_p\}_{p\in S}$ is the canonical orthonormal basis of $\ell^2(S)$, then $C_r^*(S)$ is generated by the isometries $V_p$ in $\mathcal{L}(\ell^2(S))$ given by $V_p(\varepsilon_q)=\varepsilon_{pq}$. The universal property of the full semigroup $C^*$-algebra gives a $*$-homomorphism $\lambda:C^*(S)\to C_r^*(S)$ which sends $v_p$ to $V_p$ for each $p\in S$; this is the \emph{left regular representation}.

A particularly nice class of semigroups in terms of the structure $\JJ(S)$ is the class of right LCM semigroups $S$: Recall from \cite[Lemma 3.3]{BLS1} that $\JJ(S) = \{\emptyset\} \cup \{pS \mid p \in S\}$ provided that $S$ is a right LCM semigroup with identity. Further, $\JJ(S)$ is always independent, see \cite[Corollary 3.6]{BLS1}. Moreover, $C^*(S)$ is the closed span of products of the form $v_pv_q^*$ for $p,q\in S$, see \cite[Lemma 3.11]{BLS1}.

\begin{thm}\label{thm:funct for right LCM}
Suppose  $S_1$ and $S_2$ are right LCM semigroups with identity. Let $(\tilde{v}_p)_{p \in S_1}$ and $(v_p)_{p \in S_2}$ be families of generating isometries in $C^*(S_1)$ and $C^*(S_2)$, respectively. An identity preserving semigroup homomorphism $\phi: S_1 \to S_2$ induces a $*$-homomorphism $\varphi: C^*(S_1) \to C^*(S_2)$ such that $\tilde{v}_p \mapsto v_{\phi(p)}$ if and only if $\phi$ satisfies
\begin{equation}\label{eq:funct for right LCM}
\phi(p)S_2 \cap \phi(q)S_2 = \phi(pS_1 \cap qS_1)S_2	\text{ for all } p,q \in S_1.
\end{equation}
\end{thm}
\begin{proof} Let $\{\tilde{e}_{pS_1}\}_{p\in S_1}$ and $\{e_{pS_2}\}_{p\in S_2}$ be the standard projections in $C^*(S_1)$ and $C^*(S_2)$, so $\tilde{e}_{pS_1}=\tilde{v}_{p}\tilde{v}_{p}^*$ for $p\in S_1$ and similarly for $S_2$.

We show first that \eqref{eq:funct for right LCM} is a necessary condition. Assume therefore that $\phi$ induces a $*$-homomorphism $\varphi: C^*(S_1) \to C^*(S_2)$. Given $p,q \in S_1$, we observe that $pS_1 \cap qS_1 \neq \emptyset$ forces $\phi(p)S_2 \cap \phi(q)S_2 \supset \phi(pS_1 \cap qS_1)S_2 \neq \emptyset$. Conversely, $\phi(p)S_2 \cap \phi(q)S_2 \neq \emptyset$ together with the equation
\[e_{\phi(p)S_2 \cap \phi(q)S_2} = e_{\phi(p)S_2}e_{\phi(q)S_2}
= \varphi(\tilde{e}_{pS_1})\varphi(\tilde{e}_{qS_1})
= \varphi(\tilde{e}_{pS_1} \tilde{e}_{qS_1})
= \varphi(\tilde{e}_{pS_1 \cap qS_1})\]
implies that $pS_1 \cap qS_1 \neq \emptyset$. Hence we have $\phi(p)S_2 \cap \phi(q)S_2 \neq \emptyset$ if and only if $pS_1 \cap qS_1 \neq \emptyset$. Assuming both intersections to be non-empty, the fact that $S_1$ and $S_2$ are right LCM implies that there exist $r_1 \in S_1$ and $r_2 \in S_2$ such that $\phi(p)S_2 \cap \phi(q)S_2 = r_2S_2$ and $pS_1 \cap qS_2 = r_1S_1$. It follows that $r_2S_2 \supseteq \phi(r_1)S_2$. If  $\phi(r_1)S_2$ was a proper subset of $r_2S_2$, then the image of $e_{r_2S_2}-e_{\phi(r_1)S_2}$ under the left regular representation would be non-zero. Since the assumption on $\phi$ implies that $e_{r_2S_2} = \varphi(\tilde{e}_{r_1S_1}) = e_{\phi(r_1)S_2}$, we would obtain a contradiction. Hence, $r_2 \in \phi(r_1)S_2^*$ and thus \eqref{eq:funct for right LCM} is satisfied.

Conversely, let us assume \eqref{eq:funct for right LCM} holds. We claim that $(v_{\phi(p)})_{p \in S_1}$ and $(e_{\phi(p)S_2})_{p \in S_1}$ satisfy (L1)-(L4). Since $S_1$ and $S_2$ are right LCM, it suffices to prove (L4) because (L1)-(L3) are immediate. Let $p,q \in S_1$. If $pS_1 \cap qS_1 = \emptyset$, which by \eqref{eq:funct for right LCM} is equivalent to  $\phi(p)S_2 \cap \phi(q)S_2 = \emptyset$, then
\[e_{\phi(p)S_2}e_{\phi(q)S_2} = e_{\phi(p)S_2 \cap \phi(q)S_2} = 0,\]
 as required for (L4). If $pS_1 \cap qS_1 \neq \emptyset$, then $pS_1 \cap qS_1 = rS_1$ for some $r \in S_1$. Now \eqref{eq:funct for right LCM} ensures that
\[e_{\phi(p)S_2}e_{\phi(q)S_2} = e_{\phi(p)S_2 \cap \phi(q)S_2} = e_{\phi(pS_1 \cap qS_1)S_2} = e_{\phi(r)S_2}.\]
\end{proof}

\begin{remark}\label{rem:funct for right LCM I}
Note that the set on the right hand side in \eqref{eq:funct for right LCM} is always contained in the set on the left. As we can see from the proof, \eqref{eq:funct for right LCM} is equivalent to the requirement that the map $\phi_\JJ: \JJ(S_1) \to \JJ(S_2)$ given by $\emptyset \mapsto \emptyset$ and $pS_1 \mapsto \phi(p)S_2$ is compatible with finite intersections.
In fact, as $S_1$ and $S_2$ are right LCM semigroups with identity, the image $\JJ(\phi) := \phi_\JJ(\JJ(S_1))$ is the smallest family of right ideals in $S_2$ which satisfies
\begin{enumerate}[(a)]
\item $\emptyset,S_2 \in \JJ(\phi)$; and
\item $X \in \JJ(\phi)$ and $p \in S_1$ implies $\phi(p)X$ and $\phi(p)^{-1}X \in \JJ(\phi)$.
\end{enumerate}
Therefore we interpret \eqref{eq:funct for right LCM} as a condition which characterises when $\JJ(\phi)$ is a \emph{subfamily of constructible right ideals} of $S_2$ corresponding to $\phi(S_1)$.
\end{remark}

\begin{remark}\label{rem:funct for right LCM II}
The condition \eqref{eq:funct for right LCM} is essentially the same as (b)
from \cite[Lemma 2.18]{Li1}. The apparent difference is that the similar
identity in  \cite[Lemma 2.18]{Li1} has images  under $\phi$ of right
ideals in the left-hand side as well as in the right-hand side. Since we deal
with right LCM semigroups, we can phrase the condition only using images of
right ideals in the right-hand side of \eqref{eq:funct for right LCM}. It is
striking that for right LCM semigroups there is no need to require \cite[Lemma
2.18~(a)]{Li1}, and (b) turns out to be the precise condition which is
needed to ensure that $\phi$ induces a $*$-homomorphism on the level of the
full semigroup $C^*$-algebra.
\end{remark}

\begin{prop}\label{prop:inj and surj of the induced *-hom}
Let $S_1$ and $S_2$ be right LCM semigroups with identity and suppose $\phi: S_1 \to S_2$ is a semigroup homomorphism satisfying \eqref{eq:funct for right LCM}. Then the induced $*$-homomorphism $\varphi: C^*(S_1) \to C^*(S_2)$ is surjective if and only if $\phi$ is surjective. If $\varphi$ is injective, then $\phi$ is injective. The converse holds when the left regular representation implements an isomorphism $C^*(S_i) \cong C^*_r(S_i)$ for each $i=1,2$.
\end{prop}
\begin{proof}
 Since $\phi(S_1)$ is a subsemigroup of $S_2$ and $\varphi(C^*(S_1)) \cap \{v_q \mid q \in S_2\} = \{v_{\phi(p)} \mid p \in S_1\}$, the statement about surjectivity of $\varphi$ is immediate. Assume that $\phi$ is not injective. If $p,q \in S_1, p \neq q$ are such that $\phi(p) = \phi(q)$, then $\tilde{v}_p \neq \tilde{v}_q$ but $\varphi(\tilde{v}_p) = v_{\phi(p)} = v_{\phi(q)} = \varphi(\tilde{v}_q)$, showing that $\varphi$ is not injective either.

 Finally, suppose that $\phi$ is injective and that the left regular representation implements canonical isomorphisms $C^*(S_i) \cong C^*_r(S_i)$ for $i=1,2$.  The left regular representation of $S_2$ on $\ell^2(S_2)$ restricts to a representation of $\phi(S_1)$. Compressing this by the projection onto $\phi(S_1)$ gives rise to a representation on $\ell^2(\phi(S_1)) \subset \ell^2(S_2)$. As $\phi$ is injective, this is simply a representation of $S_1$ which is unitarily equivalent to the left regular representation on $\ell^2(S_1)$. Since $C^*(S_1) \cong C^*_r(S_1)$, we deduce that this compression of $\varphi$  is injective. Hence $\varphi$ itself is injective.
\end{proof}

\noindent We shall examine \eqref{eq:funct for right LCM} more closely for right LCM semigroups built from algebraic dynamical systems in the next section.

\section{\texorpdfstring{A semigroup $C^*$-algebra description of $\AA[G,P,\theta]$}{A semigroup C*-algebra description algebraic dynamical systems}}\label{sec: semigroup description for ADS}
\noindent In this section we show that, for an algebraic dynamical system $(G,P,\theta)$, the $C^*$-algebra $\AA[G,P,\theta]$ is isomorphic to the full semigroup $C^*$-algebra of the semidirect product $\gxp$, which is a right LCM semigroup by \cite[Proposition 8.2]{BLS1}. As an application of this identification we study functoriality of the assignment $(G,P,\theta)\mapsto \AA[G,P,\theta]$ based on the findings of Section~\ref{sec: semigroup description}.
Throughout  this section we let $(G,P,\theta)$ be an algebraic dynamical system.
We start with a description of the structure of the constructible right ideals $\JJ(\gxp)$ of $\gxp$.%

\begin{notation}\label{notation: ideals}
We denote the principal right ideals $(g,p)\bigl(\gxp\bigr)$ by $X_{(g,p)}$.
\end{notation}

\begin{prop}\label{prop: structure of J}
Let $X_{(g,p)}$ and $X_{(h,q)}$ be principal right ideals of $\gxp$, for $g,h\in G$ and $p,q\in P$. Then
\[
X_{(g,p)} \cap X_{(h,q)} =
\begin{cases}
X_{(g\theta_p(k),r)} &\text{if $pP\cap qP=rP$ for some $r\in P$
and
}\\
 & \text{$g\theta_p(k)\in h\theta_q(G)$ for
some $k\in G$,}\\
\emptyset & \text{otherwise.}
\end{cases}
\]
Hence the family of constructible right ideals $\JJ(\gxp)$ is independent. If $(T_p)_{p \in P}$ is a family of transversals, then $\JJ(\gxp) = \{\emptyset\} \cup \{X_{(g,p)} \mid p\in P, g\in T_p\}$.
\end{prop}
\begin{proof}
The  formula for the intersection of principal ideals can be found in the proof of
\cite[Proposition 8.2]{BLS1}. The remaining claims follow from \cite[Lemma 3.3
and Corollary 3.6]{BLS1}.
\end{proof}

\begin{cor}\label{cor: spn for C*}
$C^*(\gxp)$ is the closed linear span of $\{v_{(g,p)}v_{(h,q)}^* \mid g,h\in G,\, p,q\in P\}$.
\end{cor}
\begin{proof}
Since $\gxp$ is right LCM, this follows from \cite[Lemma~3.11]{BLS1}.
\end{proof}

\noindent We now state the main result of this section.

\begin{thm}\label{thm: A is a semigroup algebra}
There is an isomorphism
$\varphi:\AA[G,P,\theta]\to C^*(\gxp)$ satisfying
\[
\varphi(u_g)=v_{(g,1)}\quad\text{and}\quad \varphi(s_p)=v_{(1,p)},
\]
for all $g\in G$, $p\in P$.
\end{thm}
\begin{proof}
Straightforward calculations show that $g\mapsto v_{(g,1)}$ is a unitary
representation of $G$ and
$p\mapsto v_{(1,p)}$ is an isometric representation of $P$. We show that these
representations satisfy (A1) and
(A2). Recall that $e_{X_{(g,p)}} = v_{(g,p)}v_{(g,p)}^*$. Fix $g,h \in G$ and
$p,q \in P$. Then
\[
v_{(1,p)}v_{(g,1)} \stackrel{\text{(L1)}}{=} v_{(1,p)(g,1)} =
v_{(\theta_p(g),p)} \stackrel{\text{(L1)}}{=} v_{(\theta_p(g),1)}v_{(1,p)},
\]
so (A1) holds. To get (A2), we just check that \eqref{eq: prod of projs}
holds. We have
\[
v_{(g,1)}v_{(1,p)}v_{(1,p)}^*v_{(g,1)}^* = v_{(g,p)}v_{(g,p)}^* =
e_{X_{(g,p)}}.
\]
Proposition~\ref{prop: structure of J} gives that
\[
e_{X_{(g,p)}}e_{X_{(h,q)}}=
\begin{cases}
e_{X_{(g\theta_p(k),r)}} &\text{if } pP \cap qP = rP
\text{ and } g\theta_p(k) \in
h\theta_q(G) \text{ for some $k\in G$},\\
\emptyset & \text{otherwise,}
\end{cases}
\]
which is \eqref{eq: prod of projs}. The universal property of
$\AA[G,P,\theta]$  gives therefore a homomorphism $\varphi:\AA[G,P,\theta]\to
C^*(\gxp)$ satisfying $\varphi(u_g)=v_{(g,1)}$ and $\varphi(s_p)=v_{(1,p)}$, for all
$g\in G$, $p\in P$.
To prove the result we use the universal property of $C^*(\gxp)$ to find an inverse $\varphi'$ for $\varphi$.

Consider the elements $u_gs_p, e_{(g,p)}\in \AA[G,P,\theta]$. Each $u_gs_p$
is an isometry with range projection equal to
$e_{(g,p)}$. The elements $\{u_gs_p,e_{(h,q)} \mid g,h\in G,\, p,q\in P\}\cup\{0\}$
are easily seen to satisfy (L1)--(L4): condition (L1) follows from an
application of (A1), (L2) and (L3) are immediate, and (L4) follows from
\eqref{eq: prod of projs} and the formula for intersection of ideals from
Proposition~\ref{prop: structure of J}. The universal property of $C^*(\gxp)$
now gives a homomorphism $\varphi':C^*(\gxp)\to\AA[G,P,\theta]$ which satisfies
$v_{(g,1)}\mapsto u_g$ and $v_{(1,p)}\mapsto s_p$, for all $g\in G$, $p\in P$.
Hence $\varphi$ and $\varphi'$ are mutually inverse, so $\varphi$ is an isomorphism.
\end{proof}

\noindent Let us return to the question of functoriality for $(G,P,\theta) \mapsto \AA[G,P,\theta]$ raised at the end of Section~2. Now that we have realised $\AA[G,P,\theta]$ as the semigroup $C^*$-algebra of the right LCM semigroup $\gxp$, we can appeal to \eqref{eq:funct for right LCM} which gives us a precise condition under which a morphism $(\phi_G,\phi_P): (G_1,P_1,\theta_1) \to (G_2,P_2,\theta_2)$ of algebraic dynamical systems induces a $*$-homomorphism $\varphi_G \rtimes \varphi_P: \AA[G_1,P_1,\theta_1] \to \AA[G_2,P_2,\theta_2]$.

\begin{definition}\label{def:adm morphism of ADS}
A morphism $(\phi_G,\phi_P): (G_1,P_1,\theta_1) \to (G_2,P_2,\theta_2)$
of algebraic dynamical systems is called \emph{admissible}, if, in addition to (i)--(iii) of Definition~\ref{def:morphism of ADS},
the following two conditions hold for all $p,q \in P_1$:
\begin{enumerate}
\item[(iv)] $\phi_P(p)P_2 \cap \phi_P(q)P_2 = \phi_P(pP_1 \cap qP_1)P_2$.
\item[(v)] $\phi_G^{-1}\bigl(\phi_G(G_1) \cap \theta_{2,\phi_P(p)}(G_2)\theta_{2,\phi_P(q)}(G_2)\bigr) = \theta_{1,p}(G_1)\theta_{1,q}(G_1)$ if $pP_1 \cap qP_1 \neq \emptyset$.
\end{enumerate}
\end{definition}

\begin{remark}\label{rem:adm morphism}
Let us comment briefly on these extra requirements:
\begin{enumerate}[(a)]
\item For both (iv) and (v), the inclusion $\supset$ holds for every morphism
$(\phi_G,\phi_P)$, so the statement is about the reverse inclusion.
\item Condition (iv) is nothing but \eqref{eq:funct for right LCM} for $\phi_P: P_1 \to P_2$. As long as we do not have more knowledge on $P_1$ and $P_2$, we cannot expect to get anything better than this, compare Theorem~\ref{thm:funct for right LCM}. But the existence of a semigroup homomorphism $\phi_P: P_1 \to P_2$ satisfying (iv) has structural consequences: If $P_2$ is such that the  family of principal right ideals ordered by reverse inclusion is directed, then the same is true of $P_1$. In particular, this is the case if $P_2$ is a group.
\item Condition (v) holds if $\phi_G$ is injective or the
action $\theta_1$ is by automorphisms of $G_1$.
\item Note that (v) forces $\theta_{1,p} \in \Aut(G_1)$ whenever
$\theta_{2,\phi_P(p)} \in \Aut(G_2)$. Hence a morphism
$(\phi_G,\phi_P): (G_1,P_1,\theta_1) \to (G_2,P_2,\theta_2)$ with
$P_1^* \subsetneqq \phi_P^{-1}(P_2^*)$ is not admissible.
\end{enumerate}
\end{remark}

\noindent Before we proceed to the general considerations on functoriality we look at admissible morphisms in two particular situations.

\begin{prop}\label{prop:adm for morphism - P_1 group}
Let $P_1$ be a group. Then every morphism $(\phi_G,\phi_P): (G_1,P_1,\theta_1) \to (G_2,P_2,\theta_2)$ of algebraic dynamical systems is admissible.
\end{prop}
\begin{proof}
If $P_1$ is a group, then $P_1^*=P_1$ and $\phi_P(P_1) \subseteq P_2^*$. Hence
$\phi_P(p)P_2 \cap \phi_P(q)P_2 = P_2 = \phi_P(pP_1 \cap qP_1)P_2$ for all $p,q \in P_1$,
which establishes (iv). Moreover, $\theta_1$ and $\theta_{2,\phi_P(\cdot)}$ are $P_1$-actions by group automorphisms of $G_1$ and $G_2$, respectively, so
\[\phi_G^{-1}\bigl(\phi_G(G_1) \cap \theta_{2,\phi_P(p)}(G_2)\theta_{2,\phi_P(q)}(G_2)\bigr) = G_1 = \theta_{1,p}(G_1)\theta_{1,q}(G_1)\]
for all $p,q \in P_1$. This shows (v).
\end{proof}

\begin{prop}\label{prop:adm for morphism - P_1 free cone}
Let $P_1 = \mathbb{F}_n^+$ for some $1 \leq n \leq \infty$. A morphism of algebraic dynamical systems $(\phi_G,\phi_P): (G_1,P_1,\theta_1) \to (G_2,P_2,\theta_2)$ is admissible if and only if it satisfies $\phi_G^{-1}\bigl(\phi_G(G_1) \cap \theta_{2,\phi_P(p)}(G_2)\bigr) = \theta_{1,p}(G_1)$ for all $p \in P_1$.
\end{prop}
\begin{proof}
Fix $p,q \in P_1$. Since $pP_1$ and $qP_1$ are disjoint unless $pP_1 \subset qP_1$ or $qP_1 \subset pP_1$, we can assume $q \in pP_1$, i.e., $q= pr$ for some $r \in P_1$. In this case,
\[\phi_P(p)P_2 \cap \phi_P(q)P_2 = \phi_P(p)P_2 \cap \phi_P(p)\phi_P(r)P_2 = \phi_P(q)P_2 = \phi_P(pP_1 \cap qP_1)P_2\]
shows (iv). In addition, we get $\theta_{1,p}(G_1)\theta_{1,q}(G_1) = \theta_{1,p}(G_1)$ and
\[\phi_G^{-1}\bigl(\phi_G(G_1) \cap \theta_{2,\phi_P(p)}(G_2)\theta_{2,\phi_P(q)}(G_2)\bigr)
= \phi_G^{-1}\bigl(\phi_G(G_1) \cap \theta_{2,\phi_P(p)}(G_2)\bigr),\]
so (v) is valid if and only if $\phi_G^{-1}\bigl(\phi_G(G_1) \cap \theta_{2,\phi_P(p)}(G_2)\bigr) = \theta_{1,p}(G_1)$ for all $p \in P_1$.
\end{proof}

\noindent The next two results show that admissible morphisms yield the right class of morphisms with respect to functoriality for $(G,P,\theta) \mapsto \AA[G,P,\theta]$.

\begin{lemma}\label{lem:adm char funct prop}
Suppose $(\phi_G,\phi_P): (G_1,P_1,\theta_1) \to (G_2,P_2,\theta_2)$ is a morphism
of algebraic dynamical systems. Then $\phi_G \rtimes \phi_P: G_1 \rtimes_{\theta_1} P_1 \to G_2 \rtimes_{\theta_2} P_2$ satisfies \eqref{eq:funct for right LCM} if and only if $(\phi_G,\phi_P)$ is admissible.
\end{lemma}
\begin{proof}
For convenience, let us denote $S_i := G_i \rtimes_{\theta_i}P_i$ for $i=1,2$ and $\phi:=\phi_G \rtimes \phi_P$. Pick $p,q \in P_1$. We need to show that the equality
\begin{equation}\label{eq:funct for GxP}
\phi(g,p) S_2 \cap \phi(h,q) S_2 = \phi\bigl((g,p)S_1 \cap (h,q)S_1\bigr) S_2
\end{equation}
holds for all $g,h \in G_1$ if and only if we have
\[\begin{array}{rrcll}
\text{(iv)}&\phi_P(p)P_2 \cap \phi_P(q)P_2 &\hspace*{-0.2mm}=\hspace*{-0.2mm}& \phi_P(pP_1 \cap qP_1)P_2 , &\text{and}\vspace*{2mm}\\
\text{(v)}&\phi_G^{-1}\bigl(\phi_G(G_1) \cap \theta_{2,\phi_P(p)}(G_2)\theta_{2,\phi_P(q)}(G_2)\bigr) &\hspace*{-0.2mm}=\hspace*{-0.2mm}& \theta_{1,p}(G_1)\theta_{1,q}(G_1) &\text{if } pP_1 \cap qP_1 \neq \emptyset.
\end{array}\]
It was noticed in Remark~\ref{rem:funct for right LCM I} and Remark~\ref{rem:adm morphism}~(a) that,
for all of these equations, the set on the right is contained in the set on the left. Moreover, Proposition~\ref{prop: structure of J} implies that $(g,p)S_1 \cap (h,q)S_1$ is non-empty if and only if
\[pP_1 \cap qP_1 \neq \emptyset \text{ and } g^{-1}h \in \theta_{1,p}(G_1)\theta_{1,q}(G_1).\]
Likewise, $\phi(g,p) S_2 \cap \phi(h,q) S_2$ is non-empty if and only if
\begin{equation}\label{eq:non-empty intersection GxP morphism ideal}
\phi_P(p)P_2 \cap \phi_P(q)P_2 \neq \emptyset \text{ and }
\phi_G(g^{-1}h) \in \theta_{2,\phi_P(p)}(G_2)\theta_{2,\phi_P(q)}(G_2).
\end{equation}
Now suppose \eqref{eq:funct for GxP} holds. For $g=h =1$ and $p,q \in P_1$, \eqref{eq:funct for GxP} implies $\phi_P(p)P_2 \cap \phi_P(q)P_2 = \phi_P(pP_1 \cap qP_1)P_2 $, which is (iv). Next, assume that $pP_1 \cap qP_1 \neq \emptyset$ and take $g=1, h \in \phi_G^{-1}\bigl(\phi_G(G_1) \cap \theta_{2,\phi_P(p)}(G_2)\theta_{2,\phi_P(q)}(G_2)\bigr)$. We then have $\phi(1,p) S_2 \cap \phi(h,q) S_2 \neq \emptyset$. By \eqref{eq:funct for GxP}, this implies $(1,p)S_1 \cap (h,q)S_1 \neq \emptyset$. Hence $h \in \theta_{1,p}(G_1)\theta_{1,q}(G_1)$ and (v) holds.

Conversely, suppose (iv) and (v) are satisfied. If $\phi(g,p) S_2 \cap \phi(h,q) S_2$ is empty, then there is nothing to show, so assume it is non-empty. As noted before, this means that
we have \eqref{eq:non-empty intersection GxP morphism ideal}. In this case, (v) implies $g^{-1}h \in \theta_{1,p}(G_1)\theta_{1,q}(G_1)$. Together with (iv) this yields \eqref{eq:funct for GxP} for $(g,p)$ and $(h,q)$.
\end{proof}

\begin{cor}\label{cor:functoriality for adm morphism of ADS}
A morphism $(\phi_G,\phi_P): (G_1,P_1,\theta_1) \to (G_2,P_2,\theta_2)$ of algebraic dynamical systems induces a $*$-homomorphism $\varphi_G \rtimes \varphi_P:\AA[G_1,P_1,\theta_1] \to \AA[G_2,P_2,\theta_2]$ if and only if $(\phi_G,\phi_P)$ is admissible.
\end{cor}
\begin{proof}
This follows from applying Lemma~\ref{lem:adm char funct prop}, Theorem~\ref{thm:funct for right LCM} and Theorem~\ref{thm: A is a semigroup algebra}.
\end{proof}

\section{\texorpdfstring{The $K$-theory for right LCM semigroups that are left Ore}{The K-theory for right LCM semigroups that are left Ore}}\label{sec: K-theory}

\noindent  The goal of this section is to compute the $K$-theory of $C^*_r(S)$ for a large class of right LCM semigroups. In view of Theorem~\ref{thm: A is a semigroup algebra}, we will thus obtain the $K$-theory of $\AA[G,P,\theta]$ whenever $\gxp$ is left Ore and the left regular representation from $C^*(\gxp)$ onto $C_r^*(\gxp)$ is an isomorphism. As we shall see, right LCM left Ore semigroups form a natural class to which  the results in \cite[\S 7]{CEL1} apply nicely: thus, as an outcome we will obtain  a stronger statement than just a computation of the $K$-groups.

Recall that a semigroup $S$ is \emph{right reversible} if every pair of non-empty left ideals has non-empty intersection. If $S$ is cancellative and right reversible, then it is called a \emph{left Ore semigroup}. It is well-known and an essential reason for the importance of the class of left Ore semigroups that a semigroup $S$ is left Ore if and only if it embeds into a group $\mathcal{G}$ such that $\mathcal{G} = S^{-1}S$. The standing assumption in \cite[\S 7]{CEL1} is that $S$ is a left Ore semigroup such that the family $\JJ(S)$ of constructible right ideals is independent. By \cite[Corollary 3.6]{BLS1},  every right LCM semigroup $S$ with identity satisfies that $\JJ(S)$ is independent. For the remainder of the section we assume that $S$ is a right LCM left Ore semigroup with identity, and we let $\mathcal{G}(S)$ denote its enveloping group. Recall also that $S^*$ is the group of invertible elements in $S$.

We recall now notation and facts from \cite{CEL1}.  For $q\in S$ and $X\in \JJ(S)\setminus\{\emptyset\}$, we let $q^{-1}\cdot X$ denote the subset
$\{q^{-1}r\mid r\in X\}$ of $\mathcal{G}(S)$. This should not be confused with the right ideal $q^{-1}X$ defined in Section~\ref{sec: semigroup description}. However, if $X$ is a subset of $qS$, then $q^{-1}X \subset S$ can be identified with $q^{-1}\cdot X \subset \mathcal{G}(S)$ using the canonical embedding $S\subset \mathcal{G}(S)=S^{-1}S$ because $q^{-1}\cdot X \subset 1^{-1}S$.
Let $Y$ denote the set
\[
S^{-1}\cdot (\JJ(S)\setminus\{\emptyset\}):=\{ q^{-1}\cdot X\mid q\in S, X\in \JJ(S)\setminus\{\emptyset\}\}.
\]
The group $\mathcal{G}(S)$ acts on $Y$ by left multiplication. This action can be described as follows: for $s_1^{-1}s_2 \in \mathcal{G}(S)$ and $t_1^{-1}\cdot t_2S$ with $s_i,t_i \in S$, the left Ore condition for $S$ implies that there are $r_1,r_2 \in S$ such that $s_1^{-1}s_2t_1^{-1}t_2 = r_1^{-1}r_2$ and therefore
\[s_1^{-1}s_2 \cdot (t_1^{-1}\cdot t_2S) = r_1^{-1}\cdot r_2S.\]
For  $X\in \JJ(S)\setminus\{\emptyset\}$, the stabiliser subgroup is
\[
\mathcal{G}(S)^{X}=\{g\in \mathcal{G}(S)\mid g\cdot (1^{-1}\cdot X)=1^{-1}\cdot X\}.
\]
We need to choose a complete set of representatives $\mathcal{X}$ for the quotient $\mathcal{G}(S)\backslash Y$ with the property that $\mathcal{X}\subseteq \JJ(S)\setminus\{\emptyset\}$.

\begin{lemma} The action of $\mathcal{G}(S)$ on $Y$ is transitive. In particular, $\{S\}$ is a complete set of representatives for the quotient $\mathcal{G}(S)\backslash Y$.
\end{lemma}
\begin{proof} Let $X_1=t_1S$ and $X_2=t_2S$ for $t_1, t_2\in S$ be two elements of $\JJ(S)\setminus\{\emptyset\}$, see \cite[Lemma 3.3]{BLS1}. Let $s_1,s_2\in S$. We need to find $g\in \mathcal{G}(S)$ such that $g\cdot (s_1^{-1}\cdot X_1)=s_2^{-1}\cdot X_2$. Note that $r\cdot(1^{-1}\cdot S)=1^{-1}\cdot rS$ for all $r\in S$. Hence, the choice $g=s_2^{-1}t_2t_1^{-1}s_1$ works.
\end{proof}

\begin{lemma} The stabiliser group $\mathcal{G}(S)^S$ of $S \in Y$ is $S^*$.
\end{lemma}
\begin{proof} It suffices to note that for all $s,t$ in $S$ we have
\begin{align*}
s^{-1}t\in \mathcal{G}(S)^S&\iff s^{-1}t\cdot (1^{-1}\cdot S)=(1^{-1}\cdot S)\\
&\iff 1^{-1}\cdot tS=1^{-1}\cdot sS\\
&\iff tS=sS\\
&\iff s^{-1}t\in S^*.
\end{align*}
\end{proof}

\noindent It is proved in \cite[Lemma 4.2]{CEL1} that the reduced semigroup $C^*$-algebra $C_r^*(S)$ is Morita equivalent to a certain reduced crossed product $D_r^{(\infty)}(S)\rtimes_r \mathcal{G}(S)$. Let $\{V_s\}_{s\in S}$ denote the family of isometries on $l^2(S)$ that generates $C^*_r(S)$ and consider the homomorphism
 \[
 \Psi_S:C_r^*(S^*)\to C_r^*(S),\lambda_{s^{-1}t}\mapsto V_s^*V_t \text{ for }s,t\in S
 \]
 from \cite[Lemma 7.2]{CEL1}.

We are now ready to state the main result of this subsection. With the preparation above, the proof is a direct application of  \cite[Corollary 7.4]{CEL1}. As in \cite{CEL1}, $K_\ast$ stands for the direct sum of $K_0$ and $K_1$ viewed as a $\mathbb{Z}/2\mathbb{Z}$-graded abelian group and $K^*$ is the $K$-homology.

\begin{thm}\label{thm-Kgroups-red-semigroup-algebra}
If the enveloping group $\mathcal{G}(S)$ of $S$ satisfies the Baum-Connes conjecture with coefficients in the $\mathcal{G}(S)$-algebras $c_0(Y)$ and $D_r^{(\infty)}(S)$, then $K_*(\Psi_S):K_*(C^*_r(S^*))\to K_*(C^*_r(S))$ is an isomorphism.
If, moreover, $c_0(Y)\rtimes_{r} \mathcal{G}(S)$ and $D_r^{(\infty)}(S)\rtimes_r \mathcal{G}(S)$ satisfy the UCT or $\mathcal{G}(S)$ satisfies the strong Baum-Connes conjecture with coefficients in  $c_0(Y)$ and $D_r^{(\infty)}(S)$, then $K^*(\Psi_S):K^*(C^*_r(S^*))\to K^*(C^*_r(S))$ is an isomorphism.
\end{thm}

\noindent We now apply Theorem~\ref{thm-Kgroups-red-semigroup-algebra} to the semigroups arising from algebraic dynamical systems $(G, P,\theta)$. Our main result is a considerable generalisation of \cite[Theorem 6.3.4]{CEL2}.


\begin{prop}\label{prop: semigp is left Ore}
If $(G, P,\theta)$ is an algebraic dynamical system, then the group of invertible elements $(\gxp)^*$ is isomorphic to $G\rtimes_\theta P^*$ and $\gxp$ is left Ore precisely when $P$ is left Ore.
\end{prop}
\begin{proof}
The claim about the group of units follows, for example, from \cite[Lemma 2.4]{BLS1}.
It is clear that $P$ needs to be right reversible and right cancellative whenever $\gxp$ has these two properties. In other words, $P$ itself has to be a left Ore semigroup as left cancellation is assumed throughout. So let us suppose that this $P$ is left Ore. If we have $(g_1\theta_{p_1}(h),p_1q) = (g_2\theta_{p_2}(h),p_2q)$, then right cancellation for $P$ and then for $G$ forces $p_1 = p_2$ and $g_1 = g_2$. To see that $\gxp$ is right reversible, consider the left ideals $(\gxp)(g,p)$ and $(\gxp)(h,q)$. By assumption, there exists $r \in P$ with $r=p'p=q'q$ and hence
\[
(1,q')(h,q)=(\theta_{q'}(h),r)=(\theta_{q'}(h)\theta_{p'}(g)^{-1},p')(g,p)\in (\gxp)(g,p)\cap (\gxp)(h,q).
\]
So $\gxp$ is right reversible, and hence is a left Ore semigroup.
\end{proof}

\noindent Recall that $\lambda:C^*(\gxp)\to C^*_r(\gxp)$ denotes the left regular representation and $\varphi:\AA[G,P,\theta]\to C^*(\gxp)$ is the isomorphism from Theorem~\ref{thm: A is a semigroup algebra}. We have the following consequence of Theorem~\ref{thm-Kgroups-red-semigroup-algebra}.

\begin{cor}\label{cor-Kgroups-red-semigroup-algebra}
Let $(G, P,\theta)$ be an algebraic dynamical system such that $P$ is a left Ore semigroup and let $\AA[G,P,\theta]$ be its associated $C^*$-algebra. Assume that $\lambda$ is an isomorphism.
If $\mathcal{G}(\gxp)$ satisfies the Baum-Connes conjecture with coefficients in the $\mathcal{G}(\gxp)$-algebras $c_0(Y)$ and $D_r^{(\infty)}(\gxp)$, then
\[
K_*\bigl((\lambda \circ\varphi)^{-1}\circ \Psi_{\gxp}\bigr):K_*(C^*_r(G\rtimes_\theta P^*))\to K_*(\AA[G,P,\theta])
\]
is an isomorphism.

If, moreover, $c_0(Y)\rtimes_{r} \mathcal{G}(\gxp)$ and $D_r^{(\infty)}(\gxp)\rtimes_r \mathcal{G}(\gxp)$ satisfy the UCT or $\mathcal{G}(\gxp)$ satisfies the strong Baum-Connes conjecture with coefficients in  $c_0(Y)$ and $D_r^{(\infty)}(\gxp)$, then
\[
K^*((\lambda \circ\varphi)^{-1}\circ\Psi_{\gxp}):K^*(C^*_r(G\rtimes_\theta P^*))\to K^*(\AA[G,P,\theta])
\]
is an isomorphism.
 \end{cor}

\noindent Corollary~\ref{cor-Kgroups-red-semigroup-algebra} applies whenever $\mathcal{G}(\gxp)$ satisfies the strong Baum-Connes conjecture with commutative coefficients, in particular it applies when $\mathcal{G}(\gxp)$ is amenable. To check amenability, one possible strategy is as follows: by \cite[Proposition 6.1.3]{CEL2}, there is an embedding of $\gxp$  into a semidirect product  $G^{(\infty)} \rtimes_{\theta^{(\infty)}} P^{-1}P$, where $G^{(\infty)}$ is the inductive limit of $(G,\theta_p)$ over $P$.  The universal property of $\mathcal{G}(\gxp)$ gives an injective group homomorphism $\mathcal{G}(\gxp)\hookrightarrow G^{(\infty)} \rtimes_{\theta^{(\infty)}} P^{-1}P$. Thus amenability of $\mathcal{G}(\gxp)$ will follow from that of $G^{(\infty)} \rtimes_{\theta^{(\infty)}} P^{-1}P$. For instance, the latter is amenable whenever $G$ and $P^{-1}P$ are amenable.

\section{Nica-Toeplitz algebras of product systems over right LCM semigroups}\label{sec: right LCM product systems}

\noindent Our next aim is to describe $\AA[G,P,\theta]$ as a $C^*$-algebra associated to a
product system of right-Hilbert $C^*(G)$-bimodules. As the natural product system to consider
is fibred over the right LCM semigroup $P$, we need to extend the existing
theory to build Nica-Toeplitz algebras for such product systems. In this
section we define the notions of a compactly aligned product system over a
right LCM semigroup, and of a Nica covariant representation of such a product
system. We use Parseval frames to find a condition under which a product
system over a right LCM semigroup is compactly aligned,
and we show that the Fock representation is Nica covariant.

Let $A$ be a $C^*$-algebra and $S$ a (countable, discrete) left cancellative
semigroup with identity $1$. Recall from \cite{Fow2} that a {\em
product system} of right-Hilbert $A$-bimodules over $S$ is a semigroup
$M=\bigsqcup_{p\in S}M_p$, where each $M_p$ is a right-Hilbert $A$-bimodule;
$x\otimes_A y\mapsto xy$ determines an isomorphism of $M_p\otimes_A M_q$ onto
$M_{pq}$ for all $p,q\in S$ such that $p\neq 1$; the bimodule $M_1$ is
${}_AA_A$; and the products from $M_1\times M_p\to M_p$ and $M_p\times M_1 \to
M_p$ are given by the module actions of $A$ on $M_p$.

Recall that for each $p\in S$ and each $\xi,\eta\in M_p$ the operator
$\Theta_{\xi, \eta}:M_p\to M_p$ defined by $\Theta_{\xi,\eta}(\mu):=
\xi\cdot{\langle \eta,\mu\rangle}_p$ is adjointable with
$\Theta_{\xi,\eta}^*=\Theta_{\eta,\xi}$. The space
$\KK(M_p):=\overline{\newspan}\{\Theta_{\xi,\eta} \mid \xi,\eta\in M_p\}$ is a
closed two-sided ideal in $\LL(M_p)$ called the {\em algebra of generalised compact
operators} on $M_p$. For $p,q\in S$, there is a homomorphism
$\iota_p^{pq}:\LL(M_p)\rightarrow\LL(M_{pq})$ characterised by
\begin{equation}\label{eq: char of iota map}
\iota_p^{pq}(T)(\xi\eta)=(T\xi)\eta\quad\text{for all $\xi\in M_p,\eta\in
M_q$.}
\end{equation}
In other words, $\iota_p^{pq}(T)= T \otimes \id_{M_q}$.
For $q \notin pS$ we set $\iota_p^q = 0$.

A {\em representation} $\psi$ of $M$ in a $C^*$-algebra $B$ is a map
$M\to B$ such that
\begin{enumerate}[(1)]
\item each $\psi_p:=\psi|_{M_p}:M_p\to B$ is linear,
$\psi_1:A\to B$ is a homomorphism,
\item $\psi_p(\xi)\psi_q(\eta)=\psi_{pq}(\xi\eta)$ for all $p,q\in S$,
$\xi\in M_p$, and $\eta\in M_q$; and
\item $\psi_p(\xi)^*\psi_p(\eta)=\psi_1({\langle \xi,\eta\rangle}_p)$
for all $p\in S$, and $\xi,\eta\in M_p$.
\end{enumerate}
For each $p\in S$ there is a homomorphism $\psi^{(p)}:\KK(M_p)\to B$
characterised by
\[\psi^{(p)}(\Theta_{\xi,\eta})=\psi_p(\xi)\psi_p(\eta)^*.\]

In \cite[Definition~5.7]{Fow2} Fowler introduced the notion of a compactly
aligned product system for quasi-lattice ordered pairs. We now extend
this notion to product systems over right LCM semigroups.

\begin{definition}\label{def:compactly aligned for right LCM}
A product system $M$ over a right LCM semigroup $S$ is called \emph{compactly
aligned} if for all $p,q,r \in S$ such that $pS \cap qS = rS$ and all $T_p \in
\KK(M_p), T_q \in \KK(M_q)$ we have $\iota_p^r(T_p)\iota_q^r(T_q) \in \KK(M_r)$.
\end{definition}

\begin{remark}\label{rem: cp aligned condition}
Recall that if $r$ is a right least common multiple of $p$ and $q$, then all right least common multiples of $p$ and $q$ are of the form $rx$ for $x\in S^*$. Note that $M_x$ is a Morita equivalence from $A$ to $A$ for every $x \in S^*$ because $M_x \otimes_A M_{x^{-1}} \cong {}_AA_A \cong M_{x^{-1}} \otimes_A M_x$. Hence $T \mapsto T \otimes \text{id}$ defines an adjunction between $M_x$ and $M_{x^{-1}}$ in the sense of \cite[Definition 2.17]{CCH}. Combining this with \cite[Lemma 3.7]{CCH}, we deduce that the natural isomorphism $\iota_r^{rx}: \LL(M_r) \to \LL(M_{rx})$ restricts to an isomorphism from $\KK(M_r)$ to $\KK(M_{rx})$. Thus, $\iota_p^r(T_p)\iota_q^r(T_q)$ belongs to $\KK(M_r)$ if and only if $\iota_p^{rx}(T_p)\iota_q^{rx}(T_q) \in \KK(M_{rx})$ for all $x \in S^*$.
\end{remark}

\noindent We will shortly present a class of compactly aligned product systems. Before doing so
we recall that a \emph{Parseval frame} for a right-Hilbert module $Y$ over a unital $C^*$-algebra $A$ is
a countable family $(\xi_i)_{i \in I} \subset Y$ that satisfies the \emph{reconstruction formula}
$\eta=\sum_{i \in I} \xi_i \cdot \langle \xi_i,\eta \rangle$ for all $\eta \in Y$, where the convergence is in norm in $Y$. Parseval frames are also referred to as standard normalised tight frames, which are the main objects of interest in  \cite{Frank-Larson}. For our purposes the reconstruction formula is more relevant than the equivalent Parseval-type identity in $A$ that defines a standard normalised tight frame.

\begin{lemma}\label{lem:cp aligned from cp aligned Par frames}
Let $A$ be a unital $C^*$-algebra, $S$ a right LCM semigroup with identity, and
$M=\bigsqcup_p M_p$ a product system of right-Hilbert $A$-bimodules
such that the following condition is satisfied:
\begin{enumerate}
\item[(CA)] Whenever $p,q,r \in S$ satisfy $pS \cap qS = rS$, there are
Parseval frames $(\xi_i^{(p)})_{i \in I_p}$ for $M_p$ and $(\xi_i^{(q)})_{i
\in I_q}$ for $M_q$, such that
$\iota_p^r(\Theta_{\xi_i^{(p)},\xi_i^{(p)}})\iota_q^r(\Theta_{\xi_j^{(q)},
\xi_j^{(q)}})$  is in $\KK(M_r)$  for all $i \in I_p, j \in
I_q$.
\end{enumerate}
Then $M$ is compactly aligned.
\end{lemma}
\begin{proof}
Let $p,q,r \in S$ satisfy $pS \cap qS = rS$. By linearity and continuity it suffices to show $\iota_p^r(\Theta_{\eta_1,\eta_2})\iota_q^r(\Theta_{\zeta_1,\zeta_2}) \in \KK(M_r)$ for all $\eta_1,\eta_2 \in M_p, \zeta_1,\zeta_2 \in M_q$. We claim that, due to the reconstruction formula, we have
\[\begin{array}{rcl}
\sum\limits_{i \in F}  \Theta_{\eta_1,\eta_2}\Theta_{\xi_i^{(p)},\xi_i^{(p)}}
&\stackrel{F \subset I_p \text{ finite}}{\longrightarrow}&
\Theta_{\eta_1,\eta_2},\text{ and}\\
\sum\limits_{j \in F'}
\Theta_{\xi_j^{(q)},\xi_j^{(q)}}\Theta_{\zeta_1,\zeta_2} &\stackrel{F' \subset
I_q \text{ finite}}{\longrightarrow}&  \Theta_{\zeta_1,\zeta_2},
\end{array}\]
with convergence in norm in $\KK(M_r)$. Firstly, note that it suffices to consider the second case, where the rank one operators corresponding to the elements from the Parseval frame are to the left, because the involution on $\LL(M_p)$ is continuous. We start by observing that
\[\begin{array}{lcl}
\sum\limits_{j \in F'} \Theta_{\xi_j^{(q)},\xi_j^{(q)}}\Theta_{\zeta_1,\zeta_2} -\Theta_{\zeta_1,\zeta_2}
&=& \Theta_{(\sum\limits_{j \in F'}\xi_j^{(q)}.\langle \xi_j^{(q)},\zeta_1 \rangle)-\zeta_1,\zeta_2}.
\end{array}\]
It is routine to see that $\|\Theta_{\xi_1,\xi_2}\| \leq \|\xi_1\|\cdot\|\xi_2\|$ holds for all $\xi_1, \xi_2\in M$, where $M$ is a right-Hilbert module over $A$. Using this observation, we obtain
\[\begin{array}{lcl}
\|\sum\limits_{j \in F'} \Theta_{\xi_j^{(q)},\xi_j^{(q)}}\Theta_{\zeta_1,\zeta_2} -\Theta_{\zeta_1,\zeta_2}\|
&=& \|\Theta_{(\sum\limits_{j \in F'}\xi_j^{(q)}.\langle \xi_j^{(q)},\zeta_1 \rangle)-\zeta_1,\zeta_2}\|\\
&\leq& \|\sum\limits_{j \in F'}\xi_j^{(q)}.\langle \xi_j^{(q)},\zeta_1 \rangle-\zeta_1\| \|\zeta_2\|.
\end{array}\]
Therefore $\sum_{j \in F'}\Theta_{\xi_j^{(q)},\xi_j^{(q)}}\Theta_{\zeta_1,\zeta_2}$ converges in norm to $\Theta_{\zeta_1,\zeta_2}$ as $F' \nearrow I_q$ and, likewise, $\sum_{i \in F}  \Theta_{\eta_1,\eta_2}\Theta_{\xi_i^{(p)},\xi_i^{(p)}}$ converges in norm to $\Theta_{\eta_1,\eta_2}$ as $F \nearrow I_p$.
From this we deduce
\[\begin{array}{c}
\sum\limits_{(i,j) \in F\times F'}
\iota_p^r(\Theta_{\eta_1,\eta_2})\iota_p^r(\Theta_{\xi_i^{(p)},\xi_i^{(p)}})\iota_q^r(\Theta_{\xi_j^{(q)},\xi_j^{(q)}})\iota_q^r(\Theta_{\zeta_1,\zeta_2})
 \stackrel{F\times F' \subset I_p \times I_q \text{ finite}}{\longrightarrow}
\iota_p^r(\Theta_{\eta_1,\eta_2})\iota_q^r(\Theta_{\zeta_1,\zeta_2}),
\end{array}\]
where, again, we have convergence with respect to the operator norm on $\LL(M_r)$. By condition (CA), the operator
\[\begin{array}{c} \sum\limits_{(i,j) \in F\times F'}
\iota_p^r(\Theta_{\xi_i^{(p)},\xi_i^{(p)}})\iota_q^r(\Theta_{\xi_j^{(q)},\xi_j^{(q)}})
 \end{array}\]
is in $\KK(M_r)$ for each  finite set $F\times F' \subset I_p \times I_q$. Since $\KK(M_r)$ is a norm closed ideal in $\LL(M_r)$, the result follows.
\end{proof}

\noindent To associate a $C^*$-algebra to product systems over right LCM semigroups we
now introduce the notion of Nica covariant representations of such product
systems, which is an extension of Fowler's Nica covariance
\cite[Definition~5.7]{Fow2} for
product systems of quasi-lattice ordered pairs. In the following, we let $\varphi_p: A\to  \LL(M_p)$ denote the $*$-homomorphism that implements the left action in $M_p$.

\begin{definition}\label{def:Nica covariance for right LCM}
A representation $\psi:M\to B$ of a compactly aligned product system $M$ over
a right
LCM semigroup $S$ in a $C^*$-algebra $B$ is called \emph{Nica covariant} if
for all $p,q \in S$ and
$T_p \in \KK(M_p), T_q \in \KK(M_q)$ we have
\begin{equation}\label{eq:new-Nica-cov}
\psi^{(p)}(T_p)\psi^{(q)}(T_q)=
\begin{cases}
\psi^{(r)}\big(\iota_p^r(T_p)\iota_q^r(T_q)\big) &
\text{if $pS \cap qS = rS$ for some $r \in S$,}\\
0 & \text{otherwise.}
\end{cases}
\end{equation}
\end{definition}

\begin{remark}\label{rem: Nica cov condition} An explanation is in order about this definition. Suppose we have $p,q\in S$ such that $pS \cap qS = rS$ for some $r\in S$. Let $T_p\in \KK(M_p)$ and $T_q\in \KK(M_q)$. Let $x \in S^*$, so that $rx$ is another right LCM for $p$ and $q$. Then a necessary condition for a representation $\psi:M\to B$ to be Nica covariant is that
\begin{equation}\label{eq:sufficient-condition-covariance-any-rx}
\psi^{(r)}\big(\iota_p^r(T_p)\iota_q^r(T_q)\big)=\psi^{(rx)}\big(\iota_p^{rx}(T_p)
\iota_q^{rx}(T_q)\big).
\end{equation}
Since $\iota_p^{rx}=\iota_r^{rx}\circ \iota_p^r$ and similarly for $\iota_q^{rx}$, and since $\iota_r^{rx}$ is a homomorphism on $\LL(M_r)$, we have
$$
\iota_p^{rx}(T_p)\iota_q^{rx}(T_q)=\iota_r^{rx}(\iota_p^r(T_p)\iota_q^r(T_q)).
$$
 Using this, \eqref{eq:sufficient-condition-covariance-any-rx} will follow if  $\iota_r^{rx}(\iota_p^r(T_p)\iota_q^r(T_q))\in \KK(M_{rx})$ and $\psi^{(rx)} \circ \iota_r^{rx} = \psi^{(r)}$. The first requirement is always satisfied because  by Remark~\ref{rem: cp aligned condition}, the homomorphism  $\iota_r^{rx}$  restricts to an isomorphism from $\KK(M_r)$ onto $\KK(M_{rx})$ and $\iota_p^r(T_p)\iota_q^r(T_q)\in \KK(M_r)$ by compact alignment. In the next result we present a class of product systems for which the requirement  $\psi^{(rx)} \circ \iota_r^{rx} = \psi^{(r)}$ holds for every representation $\psi$ of $M$.
\end{remark}

\begin{lemma}\label{lem: suff conds for comm diagrams}
Suppose $A$ is a unital $C^*$-algebra, $S$ is a right LCM semigroup, and $M$
is a product system of right-Hilbert $A$-bimodules over $S$. If for each $x\in S^*$ there
is an element $1_x\in M_x$ satisfying $\Theta_{1_x,1_x}=\id_{M_x}$,
then $\psi^{(rx)} \circ \iota_r^{rx} = \psi^{(r)}$ holds for every  representation $\psi$ of $M$ and all $r\in S$.
\end{lemma}
\begin{proof}
We claim that for each $r\in S$, $x\in S^*$ and $\mu,\nu\in M_r$ we have
$\iota_r^{rx}(\Theta_{\mu,\nu})=\Theta_{\mu1_x,\nu1_x}\in\KK(M_{rx})$. To see this,
let $\xi\eta\in M_{rx}$ with $\xi\in M_r$ and $\eta\in M_x$. We have
\begin{align*}
\Theta_{\mu1_x,\nu1_x}(\xi\eta) &= (\mu 1_x)\cdot \langle
\nu1_x,\xi\eta\rangle_{rx} = (\mu 1_x)\cdot \langle 1_x,\varphi_x(\langle
\nu,\xi\rangle_r)\eta\rangle_x\\
&=
\mu\big(\Theta_{1_x,1_x}(\varphi_x(\langle \nu,\xi\rangle_r)\eta)\big)\\
&= (\mu\cdot \langle \nu,\xi\rangle_r)\eta=(\Theta_{\mu,\nu}(\xi))\eta\\
&= \iota_r^{rx}(\Theta_{\mu,\nu})(\xi\eta).
\end{align*}
Then the desired conclusion follows by linearity and continuity from
\begin{align*}
\psi^{(rx)} \circ \iota_r^{rx}(\Theta_{\mu,\nu})
&=\psi^{(rx)}(\Theta_{\mu1_x,\nu1_x})=\psi_r(\mu)\psi_x(1_x)\psi_x(1_x)^*\psi_r(\nu)^*\\
&=\psi_r(\mu)\psi^{(x)}(\Theta_{1_x,1_x})\psi_r(\nu)^*\\
&=\psi_r(\mu)\psi^{(x)}(\id_{M_x})\psi_r(\nu)^*\\
&=\psi^{(r)}(\Theta_{\mu,\nu}).
\end{align*}
\end{proof}

\begin{example}
Suppose that $M$ is a product system of right-Hilbert $A$-bimodules over a right LCM
semigroup $S$. Suppose that we have a collection of automorphisms
$\{\alpha_x:x\in
S^*\}$ such that each $M_x$ is given by ${}_AA_{\alpha_x}$ with inner product
$\langle a,b\rangle_x=\alpha_x^{-1}(a^*b)$. Then the elements $1_x:=1\in A$
satisfy the hypothesis of Lemma~\ref{lem: suff conds for comm diagrams}. The
product systems we form in the next section fall in this class of examples.
\end{example}

\noindent It was observed \cite{Fow2} that a  natural representation for product systems $M$ is the \emph{Fock representation}. The details of the construction are as follows: The vector space
\[\begin{array}{c}
F(M) = \big\{ (\xi_p)_{p \in S} \mid \xi_p \in M_p, \sum\limits_{p \in S} \|\xi_p\|_p^2 < \infty\big\}
\end{array}\]
becomes a right-Hilbert $A$-bimodule when equipped with the inner product
\[\begin{array}{c}
\langle (\xi_p)_{p \in S},(\eta_p)_{p \in S}\rangle = \sum\limits_{p \in S}\langle
\xi_p,\eta_p\rangle_p,
\end{array}\]
together with left and right actions
\[\varphi_{F(M)}(a)(\xi_p)_{p \in S} = (\varphi_p(a)\xi_p)_{p \in S} \text{ and } (\xi_p)_{p \in S}\cdot a =
(\xi_p\cdot a)_{p \in S}.\]
It is routine to check that the map $\psi_F:M \to \mathcal{L}(F(M))$ given by
\[\psi_{F,p}(\xi) (\eta_q)_{q \in S} =(\chi_{pS}(q) \cdot
\xi\eta_{r})_{q\in S} \quad \text{for } \xi \in M_p\]
defines a representation of $M$, where $\chi_{pS}$ denotes the characteristic
function on $pS$ and $q=pr$. Note that $M_p \otimes_A
M_q \cong M_{pq}$ allows us to write
$\eta_{pq} = \sum_{i \in I_{pq}}\eta_{pq,i}^{'} \otimes \eta_{pq,i}^{''}$
for suitable $\eta_{pq,i}^{'} \in M_p, \eta_{pq,i}^{''} \in M_q$ (interpreted as a norm limit of finite sums if $I_{pq}$ is infinite) and we have
\[\begin{array}{c} \psi_{F,p}(\xi)^* (\eta_{pq}) = \sum\limits_{i \in I_{pq}}\varphi_{q}(\langle
\xi,\eta_{pq,i}^{'}\rangle_p)\eta_{pq,i}^{''}. \end{array}\]

\begin{prop}\label{prop:Fock rep Nica cov}
If $M$ is a compactly aligned product system over a right LCM semigroup, then the Fock representation $\psi_F$ is Nica covariant.
\end{prop}
\begin{proof}
The following is a straightforward adaptation of the proof for the
quasi-lattice ordered case from \cite[Subsection 2.3]{HLS}. Let $p,q \in S$
and $\xi,\eta \in M_p, \xi',\eta' \in M_q$. If $pS \cap qS = \emptyset$, then
$\psi_F^{(p)}(\Theta_{\xi,\eta})\psi_F^{(q)}(\Theta_{\xi',\eta'}) (\zeta_r)_{r
\in S} = 0$ for all $(\zeta_r)_{r \in S} \in F(M)$, so
$\psi_F^{(p)}(\Theta_{\xi,\eta})\psi_F^{(q)}(\Theta_{\xi',\eta'}) = 0$. So let
$r \in S$ satisfy $pS \cap qS = rS$. Since $M$ is compactly aligned we know
that $\iota_p^r(\Theta_{\xi,\eta})\iota_q^r(\Theta_{\xi',\eta'}) \in
\KK(M_r)$. If  $\iota_p^r(\Theta_{\xi,\eta})\iota_q^r(\Theta_{\xi',\eta'}) =
\Theta_{\tilde{\xi},\tilde{\eta}}$ for some $\tilde{\xi},\tilde{\eta} \in
M_r$, then it is routine to check
\[\psi_F^{(p)}(\Theta_{\xi,\eta})\psi_F^{(q)}(\Theta_{\xi',\eta'}) =
\psi_F^{(r)}(\Theta_{\tilde{\xi},\tilde{\eta}}).\]
By linearity and continuity, the required calculations extend to the general case of
$\iota_p^r(\Theta_{\xi,\eta})\iota_q^r(\Theta_{\xi',\eta'}) \in \KK(M_r)$. In
other words, we have
\[\psi_F^{(p)}(\Theta_{\xi,\eta})\psi_F^{(q)}(\Theta_{\xi',\eta'}) =
\psi_F^{(r)}(\iota_p^r(\Theta_{\xi,\eta})\iota_q^r(\Theta_{\xi',\eta'})),\]
so $\psi_F$ is Nica covariant.
\end{proof}

\begin{definition}\label{def:NT-alg for cp al over right LCM}
For a compactly aligned product system $M$ over a right LCM semigroup $S$, the
Nica-Toeplitz algebra $\mathcal{NT}(M)$ is the universal $C^*$-algebra
generated by a
Nica covariant representation of $M$. We denote the Nica covariant
representation generating $\NT(M)$ by $j_M$.
\end{definition}

\noindent Note that $\mathcal{NT}(M)$ is always nonzero due to
Proposition~\ref{prop:Fock rep Nica cov}.

\section{Nica-Toeplitz algebras for algebraic dynamical systems}\label{sec: product system for ADS}

\noindent In this section we describe the $C^*$-algebra $\AA[G,P,\theta]$ as the
Nica-Toeplitz algebra of a product system of right-Hilbert $C^*(G)$-bimodules over
the right LCM semigroup $P$. Throughout this section, let $(G,P,\theta)$ be an
algebraic dynamical system. We denote the canonical generating unitaries in
$C^*(G)$ by $\delta_g$, $g\in G$. The action $\theta$ induces an action $\alpha:P\to \End
C^*(G)$, i.e. $\alpha_p(\delta_g)=\delta_{\theta_p(g)}$. Note that injectivity passes from $\theta$ to $\alpha$, so $\alpha$ is an action by unital injective $*$-homomorphisms. We also have an action of the
opposite semigroup $P^{\op}$ by unital, positive, linear maps $L_p:C^*(G)\to C^*(G)$ given by
$L_p(\delta_g)= \chi_{\theta_p(G)}(g)\delta_{\theta_p^{-1}(g)}$.
Each $L_p$ is a transfer operator for $(C^*(G),\alpha_p)$, in the sense
that
$L_p(\alpha_p(a)b)=aL_p(b)$
for all $a,b\in C^*(G)$, and so $(C^*(G),P,\alpha,L)$ is a dynamical system in
the style of the Exel-Larsen
systems studied in \cite{BR2,Lar}.

\begin{remark}\label{rem: Hilbert bimodules for ADS}
We use the by now standard construction for associating a right-Hilbert $C^*(G)$-bimodule to
each $p \in P$: we take $C^*(G)$ as a vector space and give it
a right action $a\cdot b=a\alpha_p(b)$, and a pre-inner product $\langle a,b\rangle_p=L_p(a^*b)$. We mod out by $N_p:=\{a\in
C^*(G) \mid L_p(a^*a)=0\}$ and complete to get a right-Hilbert $C^*(G)$-module $M_p$.
Left multiplication by elements of $C^*(G)$ extends to a left action $\varphi_p$ on $M_p$
by adjointable operators.
We denote the quotient map by $\pi_p : C^*(G)\to C^*(G)/N_p$ and note that the image of $\pi_p$ is dense in $M_p$.
\end{remark}

\begin{notation}\label{not: r1 projs}
For the remainder of this section we will write $\EE_{g,p}$ for the generalised rank one projection $\Theta_{\pi_p(\delta_g),\pi_p(\delta_g)}\in\KK(M_p)$, where $p\in P$ and $g\in G$.
\end{notation}


\begin{prop}\label{prop: hilbert bimodules}
Suppose $\{M_p \mid p\in P\}$ is the collection of right-Hilbert $C^*(G)$-bimodules from
Remark~\ref{rem: Hilbert bimodules for ADS}. For each $p\in P$ and
each transversal $T_p$ of $G/\theta_p(G)$ the family $\{\EE_{(g,p)}\mid g\in T_p\}$ consists of pairwise orthogonal projections in $\KK(M_p)$. In particular, the set $\{\pi_p(\delta_g) \mid g \in T_p\}$
is an orthonormal basis for $M_p$. Furthermore, $\varphi_p(C^*(G)) \subset \KK(M_p)$
holds if and only if $G/\theta_p(G)$ is finite.
\end{prop}
\begin{proof}
For each $g,h\in T_p$ we have
\[
\langle \pi_p(\delta_g),\pi_p(\delta_h)\rangle_p = L_p(\delta_{g^{-1}h})=
\begin{cases}
1 & \text{if $g=h$,}\\
0 & \text{otherwise}.
\end{cases}
\]
Thus $\EE_{g,p}\EE_{h,p}=0$ whenever $g,h\in T_p$ are distinct. For each $h\in G$ there is a unique $g'\in T_p$ with $h \in g'\theta_p(G)$, and hence
\[\begin{array}{c}
\sum\limits_{g\in T_p}\EE_{g,p}(\pi_p(\delta_h))
=\sum\limits_{g\in T_p}\chi_{g\theta_p(G)}(h)\pi_p(\delta_h)
= \chi_{g'\theta_p(G)}(h)\pi_p(\delta_h)
= \pi_p(\delta_h).
\end{array}\]
By linearity and continuity, this means $\sum_{g\in T_p}\EE_{g,p}(m) = m$ for all $m\in M_p$, and hence
$\{\pi_p(\delta_g) \mid g \in T\}$ is an orthonormal basis for $M_p$.

For the third assertion note that $\varphi_p(1)=\sum_{g\in T_p}\EE_{g,p}$,
which, in case $T_p$ is finite, means that $\varphi_p(1)\in\KK(M_p)$, and
hence $\varphi_p(C^*(G))\subset\KK(M_p)$. If $T_p$ is infinite, then $\varphi_p(1)$ is an infinite sum (in the strict topology)
of mutually orthogonal, equivalent, non-zero projections $\EE_{g,p}$. Thus, $\varphi_p(1)$ is not compact.
\end{proof}

\noindent We can use the bimodules $M_p$ to construct a product system of
right-Hilbert $C^*(G)$-bimodules over $P$.

\begin{prop}\label{prop: product system}
Suppose $\{M_p \mid p\in P\}$ is the collection of right-Hilbert $C^*(G)$-bimodules from
Remark~\ref{rem: Hilbert bimodules for ADS}. Then the semigroup
$M:=\bigsqcup_{p \in P}M_p$ with multiplication characterised by
\[
\pi_p(\delta_g)\pi_q(\delta_h) = \pi_{pq}(\delta_{g\theta_p(h)})\quad\text{for
$p,q\in P$, $g,h\in G$},
\]
is a product system of right-Hilbert $C^*(G)$-bimodules over $P$.
\end{prop}

\begin{proof}
The map $(\pi_p(\delta_g),\pi_q(\delta_h))\mapsto
\pi_{pq}(\delta_{g\theta_p(h)})$ is bilinear, and extends to a surjective map from the
algebraic balanced tensor product $M_p\odot_{C^*(G)} M_q$ to $M_{pq}$. This map is
equivariant with respect to the left and right actions of $C^*(G)$. It also follows from the calculation
\begin{align*}
\langle \pi_p(\delta_{g_1})\otimes \pi_q(\delta_{h_1}),
\pi_p(\delta_{g_2})\otimes \pi_q(\delta_{h_2})\rangle
&= \langle \pi_q(\delta_{h_1}),\langle\pi_p(\delta_{g_1}),\pi_p(\delta_{g_2})\rangle_p\cdot\pi_q(\delta_{h_2})\rangle_q\\
&= L_q(\delta_{h_1^{-1}}L_p(\delta_{g_1^{-1}}\delta_{g_2})\delta_{h_2})\\
&=
L_q(L_p(\alpha_p(\delta_{h_1^{-1}})\delta_{g_1^{-1}}\delta_{g_2}\alpha_p(\delta_{h_2})))\\
&=
L_{pq}(\delta_{\theta_p(h_1^{-1})}\delta_{g_1^{-1}}\delta_{g_2}\delta_{\theta_p(h_2)})\\
&=\langle
\pi_{pq}(\delta_{g_1\theta_p(h_1)}),\pi_{pq}(\delta_{g_2\theta_p(h_2)})
\end{align*}
that $M_p\odot_{C^*(G)} M_q\to M_{pq}$ preserves the inner product, and hence extends
to an isomorphism $M_p\otimes_{C^*(G)} M_q\to M_{pq}$. The remaining conditions for
$M:=\bigsqcup_{p \in
P}M_p$ to be a product system are straightforward to check.
\end{proof}

\begin{prop}\label{prop:GxP prod sys ONB}
Given $p,q \in P$ and transversals $T_p,T_q$ of $G/\theta_p(G)$,
$G/\theta_q(G)$, respectively, the map $m_{p,q}:T_p \times T_q \to G$ given by $(g,h)
\mapsto g\theta_p(h)$ is injective and its image is a transversal for
$G/\theta_{pq}(G)$. In particular, $\{\pi_{pq}(\delta_{g\theta_p(h)}) \mid g
\in T_p, h \in T_q\}$ is an orthonormal basis for $M_{pq}$.
\end{prop}
\begin{proof}
Suppose we have $m_{p,q}(g_1,h_1) = m_{p,q}(g_2,h_2)$ for some $g_i \in
T_p,h_i \in T_q$. This amounts to $g_1=g_2\theta_p(h_2h_1^{-1})$, so $g_1=g_2$
as $T_p$ is a transversal. As $\theta_p$ is injective, we then immediately get
$h_1=h_2$. Thus $m_{p,q}$ is injective. Now given $g \in G$, let $g_1 \in T_p$
be the element satisfying $g \in g_1\theta_p(G)$. Next, choose $h_1 \in T_q$
such that $\theta_p^{-1}(g_1^{-1}g) \in h_1\theta_q(G)$. Then
\[(g_1\theta_p(h_1))^{-1}g = \theta_p(h_1^{-1}\theta_p^{-1}(g_1^{-1}g)) \in
\theta_p(\theta_q(G))=\theta_{pq}(G),\]
so $g \in m_{p,q}(g_1,h_1)\theta_{pq}(G)$. Thus $m_{p,q}(T_p \times T_q)$ is a
transversal for $G/\theta_{pq}(G)$.
\end{proof}

\noindent Proposition~\ref{prop:GxP prod sys ONB} can also be proven using the fact that orthonormal bases for $M_p$ and $M_q$ yield an orthonormal basis for $M_p \otimes_{C^*(G)} M_q \cong M_{pq}$, see \cite[Lemma 4.3]{LR}. 

In order to prove that the product system $M$ from Proposition~\ref{prop: product system} is compactly aligned, we need the following lemma. Within its proof, we will make multiple applications of the identity
\[
\Theta_{\pi_p(\delta_{g_1}),\pi_p(\delta_{g_2})}(\pi_p(\delta_h))
=\chi_{g_2\theta_p(G)}(h)\pi_p(\delta_{g_1g_2^{-1}h})\quad\text{for all $p\in
P$, $g_1,g_2,h\in G$}.
\]

\begin{lemma}\label{lem: prod of iota of matrix units}
For each $p,q,r \in P$ with $pP \cap qP = rP$ and
$g_1,g_2,h_1,h_2\in G$ we have
\begin{align*}
\iota_p^r(\Theta_{\pi_p(\delta_{g_1}),\pi_p(\delta_{g_2})})&
\iota_q^r(\Theta_{\pi_q(\delta_{h_1}),\pi_q(\delta_{h_2})})\\
&\qquad=
\begin{cases}
\Theta_{\pi_r(\delta_{g_1\theta_p(k)}),\pi_r(\delta_{h_2\theta_q(\ell)})}
& \text{if $g_2\theta_p(k)=h_1\theta_q(\ell)$ for some $k,\ell\in G$,}\\
0 & \text{otherwise.}
\end{cases}
\end{align*}
\end{lemma}
\begin{proof}
Choose $p',q'\in P$ with $pp'=qq'=r$. Let $T_{p'}$ and $T_{q'}$ be
transversals for
$G/\theta_{p'}(G)$ and $G/\theta_{q'}(G)$, respectively. We know from
Proposition~\ref{prop: hilbert bimodules} that $\{\pi_{p'}(\delta_i) \mid i\in
T_{p'}\}$
and $\{\pi_{q'}(\delta_j) \mid j\in T_{q'}\}$ are orthonormal bases for $M_{p'}$
and $M_{q'}$, respectively. We claim that
\begin{equation}\label{eq: decomps of iotas}
\begin{array}{c} \iota_p^r(\Theta_{\pi_p(\delta_{g_1}),\pi_p(\delta_{g_2})}) = \sum\limits_{i\in
T_{p'}}\Theta_{\pi_r(\delta_{g_1\theta_p(i)}),\pi_r(\delta_{g_2\theta_p(i)})}
\end{array}
\end{equation}
and
\begin{equation}\label{eq: decomps of iotas p2}
\begin{array}{c} \iota_q^r(\Theta_{\pi_q(\delta_{h_1}),\pi_q(\delta_{h_2})}) = \sum\limits_{j\in
T_{q'}}\Theta_{\pi_r(\delta_{h_1\theta_q(j)}),\pi_r(\delta_{h_2\theta_q(j)})}.
\end{array}
\end{equation}
To see that \eqref{eq: decomps of iotas} holds, fix $s\in G$. Using \eqref{eq:
char of iota map} we get
\begin{align*}
\iota_p^r(\Theta_{\pi_p(\delta_{g_1}),\pi_p(\delta_{g_2})})(\pi_r(\delta_s))
&=
\Theta_{\pi_p(\delta_{g_1}),\pi_p(\delta_{g_2})}(\pi_p(\delta_s))
\pi_{p'}(\delta_1)\\
&=
\chi_{g_2\theta_p(G)}(s)\pi_p(\delta_{g_1g_2^{-1}s})
\pi_{p'}(\delta_1)\\
&=
\chi_{g_2\theta_p(G)}(s)\pi_r(\delta_{g_1g_2^{-1}s}).
\end{align*}
For $i\in T_{p'}$ we have
\[
\Theta_{\pi_r(\delta_{g_1\theta_p(i)}),\pi_r(\delta_{g_2\theta_p(i)})}(\pi_r(\delta_s))
 =
\chi_{g_2\theta_p(i)\theta_r(G)}(s)\pi_r(\delta_{g_1g_2^{-1}s}).
\]
Since $g_2\theta_p(G) = \bigsqcup_{i \in T_{p'}} g_2\theta_p(i)\theta_r(G)$, we know that
$s$ belongs to $g_2\theta_p(G)$ if and only if there is an $i\in T_{p'}$ with
$s\in g_2\theta_p(i)\theta_r(G)$. Note that $i$ is uniquely determined in this case. It follows
that $\iota_p^r(\Theta_{\pi_p(\delta_{g_1}),\pi_p(\delta_{g_2})})$ and
$\sum_{i\in
T_{p'}}\Theta_{\pi_r(\delta_{g_1\theta_p(i)}),\pi_r(\delta_{g_2\theta_p(i)})}$
agree on the image of $\pi_r$, and hence on $M_r$. So \eqref{eq: decomps of
iotas} holds. A similar argument gives \eqref{eq: decomps of iotas p2}.

Now observe that, for all $i \in T_{p'},j \in T_{q'}$, we have
\[\begin{array}{l}
\Theta_{\pi_r(\delta_{g_1\theta_p(i)}),\pi_r(\delta_{g_2\theta_p(i)})}
\Theta_{\pi_r(\delta_{h_1\theta_q(j)}),\pi_r(\delta_{h_2\theta_q(j)})}\vspace*{2mm}\\
\hspace*{40mm}= \chi_{g_2\theta_p(i)\theta_r(G)}(h_1\theta_q(j))
\Theta_{\pi_r(\delta_{g_1g_2^{-1}h_1\theta_q(j)}),\pi_r(\delta_{h_2\theta_q(j)})}.
\end{array}\]
It then follows from \eqref{eq: decomps of iotas} and \eqref{eq: decomps of
iotas p2} that
$\iota_p^r(\Theta_{\pi_p(\delta_{g_1}),\pi_p(\delta_{g_2})})\iota_q^r(\Theta_{\pi_q(\delta_{h_1}),\pi_q(\delta_{h_2})})
 \neq 0$ if and only if
\[
h_1\theta_q(j)\in g_2\theta_p(i)\theta_r(G)\text{ for some $i\in T_{p'}$,
$j\in T_{q'}$.}
\]
As $r \in pP$, we have $\theta_p(G)\theta_r(G) = \theta_p(G)$, so the equation from above is equivalent to $g_2^{-1}h_1 \in \theta_p(G)\theta_q(G)$. This in turn is equivalent to the
existence of $k,\ell \in G$ satisfying $g_2\theta_p(k) = h_1\theta_q(\ell)$.
In this case, the summand of
$\iota_p^r(\Theta_{\pi_p(\delta_{g_1}),\pi_p(\delta_{g_2})})\iota_q^r(\Theta_{\pi_q(\delta_{h_1}),\pi_q(\delta_{h_2})})$
 corresponding to the unique pair $(i,j)$ with $k \in i\theta_{p'}(G)$ and
$\ell \in j\theta_{q'}(G)$ is non-zero.

Now suppose $k_1,k_2,\ell_1,\ell_2 \in G$ satisfy $g_2\theta_p(k_n) = h_1\theta_q(\ell_n), n=1,2$. We need to show that $k_1^{-1}k_2 \in \theta_{p'}(G)$ and $\ell_1^{-1}\ell_2 \in \theta_{q'}(G)$. Rewriting the two equations as
\[\theta_p(k_1)\theta_q(\ell_1^{-1}) = g_2^{-1}h_1 = \theta_p(k_2)\theta_q(\ell_2^{-1})\]
gives $\theta_p(k_1^{-1}k_2) = \theta_q(\ell_1^{-1}\ell_2)$. Since $\theta$ respects the order, we conclude that $\theta_p(k_1^{-1}k_2) \in \theta_p(G) \cap \theta_q(G) = \theta_r(G)$ and, similarly, $\theta_q(\ell_1^{-1}\ell_2) \in \theta_r(G)$. Using that $\theta_p$ and $\theta_q$ are injective and $pp' = qq' = r$, we conclude that $k_1^{-1}k_2 \in \theta_{p'}(G)$ and $\ell_1^{-1}\ell_2 \in \theta_{q'}(G)$.
\end{proof}

\begin{remark}\label{rem: iota and projection Thetas}
Note that if we apply Lemma~\ref{lem: prod of iota of matrix units} to
$p,q,r \in P$ with $pP\cap qP=rP$, $g=g_1=g_2\in G$ and $h=h_1=h_2\in G$, we
get the identity
\begin{equation}\label{eq: iota and projection Thetas}
\iota_p^r(\EE_{g,p})\iota_q^r(\EE_{h,q})=
\begin{cases}
\EE_{g\theta_p(k),r} & \text{if $g\theta_p(k)\in h\theta_q(G)$ for some $k\in G$,}\\
0 & \text{otherwise.}
\end{cases}
\end{equation}
\end{remark}

\begin{prop}\label{prop: compactly aligned}
The product system $M$ from Proposition~\ref{prop: product system} is compactly aligned.
\end{prop}
\begin{proof}
This follows immediately from Lemma~\ref{lem:cp aligned from cp aligned Par
frames} and \eqref{eq: iota and projection Thetas}.
\end{proof}

\noindent We can now state the main result of this section. Recall from
Section~\ref{sec: right LCM product systems} that $j_M$ denotes to the
universal Nica covariant representation of a compactly aligned product system $M$; for each
$p\in P$ we denote by $j_p$ the restriction $j_M|_{M_p}$.

\begin{thm}\label{thm: isom NT and A}
Let $M$ be the product system given in Proposition~\ref{prop: product system}. There is an
isomorphism $\varphi: \AA[G,P,\theta]\to \NT(M)$ satisfying
\[
\varphi(u_g)=j_1(\delta_g)\quad \text{and}\quad \varphi(s_p)=j_p(\pi_p(1)),
\]
for all $g\in G$, $p\in P$.
\end{thm}

\noindent We will use the following result which provides a characterisation of
Nica covariance in terms of the generalised rank one projections $\EE_{g,p}$.

\begin{prop}\label{prop:GxP characterising Nica covariance}
Let $M$ be the product system given in Proposition~\ref{prop: product system}.
A representation $\psi$ of $M$ is
Nica covariant if and only if for all $g,h \in G$ and $p,q \in P$ we have
\begin{equation}\label{eq: char of Nica cov}
\psi^{(p)}(\EE_{g,p})\psi^{(q)}(\EE_{h,q}) =
\begin{cases}
\psi^{(r)}(\EE_{g\theta_p(k),r}) & \text{if $pP\cap qP=rP$ for some $r\in P$
and
}\\
 & \text{$g\theta_p(k)\in h\theta_q(G)$ for
some $k\in G$,}\\
0 & \text{otherwise.}
\end{cases}
\end{equation}
\end{prop}

\begin{proof}
If $\psi$ is Nica covariant, then \eqref{eq: char of Nica cov} follows
immediately from \eqref{eq: iota and projection Thetas}. For the converse
direction, we note that, by linearity and continuity, it suffices to check
\eqref{eq:new-Nica-cov} for $\Theta_{\pi_p(\delta_{g_1}),\pi_p(\delta_{g_2})}$
and $\Theta_{\pi_q(\delta_{h_1}),\pi_q(\delta_{h_2})}$ for all $p,q \in P$ and
$g_i,h_i \in G$. We have
\begin{align*}
&\psi^{(p)}(\Theta_{\pi_p(\delta_{g_1}),\pi_p(\delta_{g_2})})
\psi^{(q)}(\Theta_{\pi_q(\delta_{h_1}),\pi_q(\delta_{h_2})})\\
&\hspace{4.5cm}
=\psi^{(p)}(\Theta_{\pi_p(\delta_{g_1}),\pi_p(\delta_{g_2})}
\EE_{g_2,p})
\psi^{(q)}(\EE_{h_1,q}
\Theta_{\pi_q(\delta_{h_1}),\pi_q(\delta_{h_2})})\\
&\hspace{4.5cm}
=\psi^{(p)}(\Theta_{\pi_p(\delta_{g_1}),\pi_p(\delta_{g_2})})
\psi^{(p)}(\EE_{g_2,p})
\psi^{(q)}(\EE_{h_1,q})
\psi^{(q)}(\Theta_{\pi_q(\delta_{h_1}),\pi_q(\delta_{h_2})}),
\end{align*}
which we know by \eqref{eq: char of Nica cov} is equal to
\begin{equation}\label{eq: midway proj}
\psi^{(p)}(\Theta_{\pi_p(\delta_{g_1}),\pi_p(\delta_{g_2})})
\psi^{(r)}(\EE_{g_2\theta_p(k),r})
\psi^{(q)}(\Theta_{\pi_q(\delta_{h_1}),\pi_q(\delta_{h_2})})
\end{equation}
if $pP\cap qP=rP$ for some $r\in P$ and $g_2\theta_p(k)\in h\theta_q(G)$ for
some $k\in G$, and is zero, otherwise. We immediately see from \eqref{eq: char of Nica cov} and Lemma~\ref{lem: prod of iota of matrix units} that
\eqref{eq:new-Nica-cov} is satisfied when $pP\cap qP=\emptyset$ or $g_2^{-1}h_1 \notin \theta_p(G)\theta_q(G)$.

So suppose $pP\cap qP=rP$ for some $r\in P$ and $g_2\theta_p(k)\in h_1\theta_q(G)$
for some $k\in G$. Let $p',q'\in P$ with $pp'=qq'=r$, and $\ell\in G$ with
$g_2\theta_p(k)=h_1\theta_q(\ell)$. The expression in \eqref{eq: midway
proj} can be written as
\[
\psi_p(\pi_p(\delta_{g_1}))\psi_p(\pi_p(\delta_{g_2}))^*
\psi^{(r)}(\EE_{g_2\theta_p(k),r})
\psi_q(\pi_q(\delta_{h_1}))\psi_q(\pi_q(\delta_{h_2}))^*.
\]
Rewriting the middle term gives
\begin{align*}
\psi^{(r)}(\EE_{g_2\theta_p(k),r})&=
\psi_r(\pi_r(\delta_{g_2\theta_p(k)}))\psi_r(\pi_r(\delta_{g_2\theta_p(k)}))^*\\
&
=
\psi_r(\pi_r(\delta_{g_2\theta_p(k)}))\psi_r(\pi_r(\delta_{h_1\theta_q(\ell)}))^*\\
&
= \psi_p(\pi_p(\delta_{g_2}))\psi_{p'}(\pi_{p'}(\delta_k))
\psi_{q'}(\pi_{q'}(\delta_{\ell}))^*
\psi_q(\pi_q(\delta_{h_1}))^*.
\end{align*}
Using $\psi_1(\langle\pi_p(\delta_{g_2}),
\pi_p(\delta_{g_2})\rangle_p)=1=\psi_1(\langle\pi_q(\delta_{h_1}),
\pi_q(\delta_{h_1})\rangle_q)$, we turn the expression in \eqref{eq: midway proj} into
\begin{align*}
\psi_p(\pi_p(\delta_{g_1}))\psi_{p'}(\pi_{p'}(\delta_k))
\psi_{q'}(\pi_{q'}(\delta_{\ell}))^*
\psi_q(\pi_q(\delta_{h_2}))^*&= \psi_r(\pi_r(\delta_{g_1\theta_p(k)}))
\psi_r(\pi_r(\delta_{h_2\theta_q(\ell)}))\\
&= \psi^{(r)}(
\Theta_{\pi_r(\delta_{g_1\theta_p(k)}),\pi_r(\delta_{h_2\theta_q(\ell)})}).
\end{align*}
But we know from Lemma~\ref{lem: prod of iota of matrix units} that
\[
\psi^{(r)}\big(\iota_p^r(\Theta_{\pi_p(\delta_{g_1}),\pi_p(\delta_{g_2})})
\iota_q^r(\Theta_{\pi_q(\delta_{h_1}),\pi_q(\delta_{h_2})})\big)=
\psi^{(r)}(
\Theta_{\pi_r(\delta_{g_1\theta_p(k)}),\pi_r(\delta_{h_2\theta_q(\ell)})}).
\]
So we get
\[
\psi^{(p)}(\Theta_{\pi_p(\delta_{g_1}),\pi_p(\delta_{g_2})})
\psi^{(q)}(\Theta_{\pi_q(\delta_{h_1}),\pi_q(\delta_{h_2})})
=\psi^{(r)}\big(\iota_p^r(\Theta_{\pi_p(\delta_{g_1}),\pi_p(\delta_{g_2})})
\iota_q^r(\Theta_{\pi_q(\delta_{h_1}),\pi_q(\delta_{h_2})})\big),
\]
and hence \eqref{eq:new-Nica-cov} is satisfied.
\end{proof}

\noindent The conclusion we draw from Proposition~\ref{prop:GxP characterising
Nica covariance} is that, for the product system $M$, it suffices to check
Nica covariance on rank one projections coming from suitable orthonormal
bases.

\begin{proof}[Proof of Theorem~\ref{thm: isom NT and A}]
The proof follows closely that of Theorem~\ref{thm: A is a semigroup
algebra}. The map $g\mapsto j_1(\delta_g)$ is a unitary representation of
$G$ in $\NT(M)$, and $p\mapsto j_p(\pi_p(1))$ is a representation of $P$
by isometries in $\NT(M)$. We have
\[
j_p(\pi_p(1))j_1(\delta_g) = j_p(\pi_p(\delta_{\theta_p(g)})) =
j_1(\delta_{\theta_p(g)})j_p(\pi_p(1)),
\]
which is (A1). Instead of (A2), we show \eqref{eq: prod of projs}, see Lemma~\ref{lem: A2 and projs}. First note that
$j_1(\delta_g)j_p(\pi_p(1))j_p(\pi_p(1))^*j_1(\delta_g)^*=j^{(p)}(\EE_{g,p})$.
Now \eqref{eq: prod of projs} follows from applying \eqref{eq: char of Nica cov} to the representation $j$. So the universal property of
$\AA[G,P,\theta]$ gives a homomorphism $\varphi:
\AA[G,P,\theta]\to \NT(M)$ satisfying $\varphi(u_g)=j_1(\delta_g)$ and
$ \varphi(s_p)=j_p(\pi_p(1))$, for all $g\in G$, $p\in P$. This homomorphism
is surjective because each $j_p(\pi_p(\delta_g))=\varphi(u_gs_p)$ is in the
range of $\varphi$.

It is routine to show that $\pi_p(\delta_g) \mapsto u_gs_p$ defines a
representation of the product system $M$, and \eqref{eq: prod of projs} gives
\eqref{eq: char of Nica cov}, which is Nica covariance. The induced
homomorphism $\psi:\NT(M)\to \AA[G,P,\theta]$ is inverse to $\varphi$, and
hence $\varphi$ is an isomorphism.
\end{proof}

\noindent We note the following immediate consequence:

\begin{cor}\label{cor: C*(P) as NT-alg for right LCM P}
If $S$ is a right LCM semigroup, then $C^*(S)$ is isomorphic to the Nica-Toeplitz algebra $\mathcal{NT}(M)$ of the product system of right-Hilbert bimodules $M$ over $S$ with $M_p = \C$ for all $p \in S$.
\end{cor}
\begin{proof}
The triple $(\{1\},S,\id)$ defines an algebraic dynamical system and
$\{1\} \rtimes_{\id} S \cong S$, so Theorem~\ref{thm: A is a semigroup algebra} and
Theorem~\ref{thm: isom NT and A} imply
\[C^*(S) \cong C^*(\{1\} \rtimes_{\id} S) \cong \AA[\{1\},S,\id] \cong \mathcal{NT}(M)\]
for the product system $M$ from Proposition~\ref{prop: product system} associated to $(\{1\},S,\id)$.
\end{proof}

\section{Examples}\label{sec: examples}
\noindent Observing that {algebraic dynamical systems} are a natural
generalisation of {irreversible algebraic dynamical systems} as
introduced in \cite[Definition 1.5]{Sta1}, we already have various examples at
our disposal, namely \cite[Examples 1.8--1.11,1.14]{Sta1}.

\begin{remark}\label{rem:A vs O for IADS}
For each irreversible algebraic dynamical system $(G,P,\theta)$, the
$C^*$-algebra $\AA[G,P,\theta]$ can be viewed as the natural Toeplitz
extension of the $C^*$-algebra $\OO[G,P,\theta]$ from \cite{Sta1}. In fact,
$\AA[G,P,\theta]$ is the Nica-Toeplitz algebra of the product system $M$
associated to $(G,P,\theta)$ according to Theorem~\ref{thm: isom NT and A} and
$\OO[G,P,\theta]$ is the Cuntz-Nica-Pimsner algebra of $M$ if $(G,P,\theta)$
is of finite type, see \cite[Theorem 5.9]{Sta1}.
\end{remark}

\noindent But the class of algebraic dynamical systems is much larger:

\begin{example}\label{ex:group C*-algs}
Suppose $P$ is a discrete group acting on another discrete group $G$ by
automorphisms $\theta_p, p \in P$. Then $(G,P,\theta)$ is an algebraic
dynamical system and $\AA[G,P,\theta] = C^*(G \rtimes_\theta P) = C^*(G)
\rtimes P$. In particular, we can take $P = \{1\}$ and obtain the full group
$C^*$-algebra for $G$.
\end{example}

\begin{example}\label{ex:right LCM sgp C*-algs}
Similar to the last part in Example~\ref{ex:group C*-algs}, if we let
$G=\{1\}$, then any right LCM semigroup $P$ acts on $G$ by the trivial action
and $\AA[G,P,\theta]$ is nothing but the full semigroup $C^*$-algebra of $P$.
\end{example}

\begin{remark}\label{rem:ADS vs right LCM}
Since we know that, conversely, $\gxp$ is a right LCM semigroup for every
algebraic dynamical system $(G,P,\theta)$, we see that studying algebraic
dynamical systems is a priori more sensible than studying right LCM
semigroups. The difference is that an algebraic dynamical system comes
equipped with a decomposition of the right LCM semigroup as a semidirect
product of a group by a right LCM semigroup. A way of extracting a specific semidirect
product description for certain cancellative semigroups out of their semigroup
structure is presented in \cite[Proposition 2.11]{BLS1}.
\end{remark}

\begin{example}\label{ex: add boundary quotient}
The semigroup $\N^\times$ is a right LCM semigroup which acts upon $\Z$ by multiplication in
an order preserving way. Hence we get an algebraic dynamical system
$(\Z,\N^\times,\cdot)$. By
Theorem~\ref{thm: A is a semigroup algebra}, we have $\AA[\Z,\N^\times,\cdot]
\cong C^*(\Z \rtimes \N^\times)$, and by Theorem~\ref{thm: isom NT and A}, they both are isomorphic to
a Nica-Toeplitz algebra for a product system of right-Hilbert $C^*(\Z)$-modules over $\N^\times$. It was observed in \cite[Example 3.9]{BaHLR} and in \cite{HLS0}
that the latter $C^*$-algebra is isomorphic to the additive boundary quotient
$\TT_{\text{add}}(\N \rtimes \N^\times)$ of $\TT(\N \rtimes \N^\times)$, the
Toeplitz algebra of the affine semigroup over the natural numbers. We note
that $(\Z,\N^\times,\cdot)$ is in fact an irreversible algebraic dynamical
system and $\OO[\Z,\N^\times,\cdot] \cong \QQ_\N$.
\end{example}

\noindent Example~\ref{ex: add boundary quotient} yields an alternative
perspective on the boundary quotient diagram of \cite{BaHLR} which seems
well-suited for generalisations to other cases that are similar to $\N \subset
\Z$. The next example provides a reasonable framework for such a task.

\begin{example}\label{ex:ring of integers}
Suppose $R$ is the ring of integers in a number field $K$ and let $(r)$ denote
the principal ideal generated by $r \in R$. If we take $G=R$ (equipped with
addition) and $P \subset R^\times$ is a subsemigroup satisfying $(p) \cap (q)
\in \{ (r) \mid r \in P\} \text{ for all } p,q \in P,$ then $(G,P,\cdot)$ is
an algebraic dynamical system. Note that the intersection condition is automatically
satisfied if $R$ is a principal ideal domain, i.e. the class number of $K$ is
one. In this case, we recover the natural extension $C^*(R \rtimes R^\times)$
of the ring $C^*$-algebra $\mathfrak{A}[R]$ from \cite{CL}.
\end{example}

\begin{remark}\label{rem:ring of integers - towards the general case}
Even though there are examples of semigroups $P$ as in Example~\ref{ex:ring of
integers} for number fields $K$ with class number bigger than one, it would be
convenient to have a systematic treatment which is available for arbitrary
subsemigroups of $R^\times$ for arbitrary number fields $K$. This asks for a
relaxation of the condition
\[\theta_p(G) \cap \theta_q(G) = \theta_r(G) \text{ whenever } pP \cap qP =
rP\]
to the constraint
\[ [\theta_p(G) \cap \theta_q(G) : \theta_r(G)] < \infty \text{ whenever } pP \cap qP \supset rP \text{ and $r$ is minimal}\]
in the sense that $pP \cap qP \supset r'P$ and $r'P \supset rP$ imply $r'P=rP$. We would like to point out that this condition is somewhat reminiscent of Spielberg's notion of {finite alignment} for
categories of paths, see \cite[Section 3]{Spie}.
\end{remark}

\noindent Let us end with two basic constructions for algebraic dynamical
systems, which have been described in \cite[Examples 1.12 and 1.13]{Sta1} for
the irreversible case. The first one is built by shifting some suitable group
$G_0$ along a right LCM semigroup $P$.

\begin{prop}\label{prop:shift over right LCM}
Let $P$ be a countable, right LCM semigroup with identity and $G_0$ a
countable group. Then $P$ admits a shift action $\theta$ on $G:=\bigoplus_P G_0$
and $(G,P,\theta)$ is an algebraic dynamical system.
\end{prop}
\begin{proof}
The action $\theta$ is given by $(\theta_p((g_q)_{q \in P}))_r =
\chi_{pP}(r)~g_{p^{-1}r} \text{ for all } p,r \in P$. It is apparent that
$\theta_p$ is an injective group endomorphism for all $p \in P$. Observing
that $\theta_p(G) = 0 \oplus \bigl(\bigoplus_{q \in pP} G_0\bigr)$, we see that $\theta$
also respects the order on $P$. Hence we get an algebraic dynamical system in
the sense of Definition~\ref{def:ADS}.
\end{proof}

\begin{example}\label{ex:Toeplitz algebra for right LCM}
If we apply the construction of Proposition~\ref{prop:shift over right LCM} to
$G_0 = \Z/n\Z$ and $P = \N$, the resulting algebra $\AA[G,P,\theta]$ is
canonically isomorphic to the Toeplitz extension $\TT_n$ of the Cuntz algebra
$\OO_n$, compare \cite[Example 3.30]{Sta1}.
\end{example}

\begin{example}\label{ex:ADS as sum of ADS}\label{ex:strong ind with inf index}
Suppose we are given a family $(G^{(i)},P,\theta^{(i)})_{i\in I}$ of algebraic
dynamical systems. Then we can consider $G := \bigoplus_{i \in I} G^{(i)}$ and
let $P$ act on $G$ component-wise, i.e. $\theta_p (g_i)_{i \in I} :=
(\theta^{(i)}_p(g_i))_{i \in I}$. In this way we get a new algebraic dynamical
system $(G,P,\theta)$.
\end{example}

\begin{remark}\label{rem:decomposition of ADS}
It is also possible to decompose an algebraic dynamical system $(G,P,\theta)$.
Let $G = \bigoplus_{i \in I} G_i$ be a decomposition of $G$ into
$\theta$-invariant subgroups, i.e. $\theta_p(G_i) \subset G_i$ for all $i \in
I$. Then $(G_i,P,\theta_{|G_i})$ defines an algebraic dynamical system for
each $i \in I$. Indeed, it is easy to see that the subsystem
$(G_i,P,\theta_{|G_i})$ inherits the required properties from $(G,P,\theta)$.
Besides, we note that the procedure from Example~\ref{ex:ADS as sum of ADS} is
inverse to this method of decomposing.
\end{remark}

\begin{example}\label{ex:nonisom twists by -1 I}
For $G = \Z$, the injective group endomorphisms are given by multiplication with elements from $\Z^\times$. For $\PP \subset \Z^\times$, let $|\PP\rangle$ denote the multiplicative subsemigroup of $\Z^\times$ with identity generated by $\PP$. Let us consider subsemigroups $P_1 = | 2,3\rangle$, $P_2 = | -2,3\rangle$, $P_3 = | 2,-3\rangle$ and $P_4 = | -2,-3\rangle$.  Each of them gives rise to an irreversible algebraic dynamical system
$(\Z,P_i,\cdot)$. It is an instructive exercise to show that the
corresponding semidirect products $\Z \rtimes P_i$ are mutually non-isomorphic. However, if we
add $-1$ as a generator to each of $P_i$, we always arrive at $P_5 = |-1,2,3\rangle$.
Moreover, the corresponding embedding $\Z \rtimes P_i \subset \Z \rtimes P_5$ yields injective,
admissible morphisms from $(\Z,P_i,\theta)$ into $(\Z,P_5,\cdot)$. Hence, by
Corollary~\ref{cor:functoriality for adm morphism of ADS} there are canonical $*$-homomorphisms
from $\AA[\Z,P_i,\theta]$ to $\AA[\Z,P_5,\theta]$ for $i=1,\dots, 4$.

In this example the left regular representation implements an isomorphism between the full and the reduced semigroup $C^*$-algebras for $\Z \rtimes P_i$ for each $i=1,\dots, 5$, see \cite[Example 6.3]{BLS1} or \cite{Li1}.
Thus by Proposition~\ref{prop:inj and surj of the induced *-hom}, the
$C^*$-algebras $\AA[\Z,P_i,\cdot], i =1,\dots,4$ are proper subalgebras of $\AA[\Z,P_5,\cdot]$ in a natural way. Note
that the families of constructible right ideals $\JJ(\Z \rtimes P_i)$ are isomorphic for all five
semidirect products. Thus, their diagonal subalgebras are canonically
isomorphic, see \cite{Li1}. Since this identification is compatible with translation by $\Z$,
we also have an isomorphism on the level of the core subalgebras $\FF \cong
\DD \rtimes \Z$, a fact that follows from \cite[Definition 3.12]{BLS1} and \cite[Theorem A.5]{Sta1}.

It is therefore natural to ask if, and in how far the
$C^*$-algebras $\AA[\Z,|2,3\rangle,\cdot], \AA[\Z,|-2,3\rangle,\cdot],
\AA[\Z,|2,-3\rangle,\cdot], \AA[\Z,|-2,-3\rangle,\cdot]$ and
$\AA[\Z,|-1,2,3\rangle,\cdot]$ are different. Corollary~\ref{cor-Kgroups-red-semigroup-algebra}
shows that $K_\ast(\AA[\Z,|P_i\rangle,\cdot])$ is isomorphic to $\Z \oplus \Z$
for $i=1, \dots, 4$, and that
\[K_*\bigl(\AA[\Z,|P_5\rangle,\cdot]\bigr) = K_*\bigl(C^*(\Z \rtimes \Z/2\Z)\bigr).\]
Thus, $K$-theory distinguishes  $\AA[\Z,|P_5\rangle,\cdot]$ from any of the  $C^*$-algebras
$\AA[\Z,|P_i\rangle,\cdot]$ for $1, \dots, 4$. However, finer invariants are needed in
order to determine the exact relationship between those four $C^*$-algebras.
\end{example}

\end{document}